\documentclass[11pt,longbibliography,reqno]{amsart}

\usepackage{fullpage}

\usepackage{amsmath}
\usepackage{amsthm}
\usepackage{amssymb}
\usepackage{amsfonts}
\usepackage{mathtools}
\usepackage{paralist}
\usepackage[colorlinks=true,linkcolor=blue,citecolor=blue,urlcolor=blue,breaklinks]{hyperref}
\usepackage{url}
\usepackage{cases}
\usepackage{enumitem}
\usepackage{empheq}
\usepackage{subcaption}
\usepackage{graphicx,epstopdf}

\allowdisplaybreaks
\mathtoolsset{showonlyrefs}

\numberwithin{equation}{section}

\theoremstyle{plain}
\newtheorem{thm}{Theorem}[section]

\newtheorem{prop}[thm]{Proposition}
\newtheorem{hyp}{Hypothesis}

\theoremstyle{definition}
\newtheorem{defn}[thm]{Definition}

\theoremstyle{remark}
\newtheorem{rem}[thm]{Remark}

\newcommand{\be}{\begin{equation}}
\newcommand{\ee}{\end{equation}}
\newcommand{\bfig}{\begin{figure}}
\newcommand{\efig}{\end{figure}}
\newcommand{\bt}{\begin{table}}
\newcommand{\et}{\end{table}}
\newcommand{\bc}{\begin{center}}
\newcommand{\ec}{\end{center}}
\newcommand{\ba}{\begin{array}}
\newcommand{\ea}{\end{array}}
\newcommand{\bes}{\begin{equation*}}
\newcommand{\ees}{\end{equation*}}

\newcommand{\mt}[1]{\mathrm{#1}}
\newcommand{\wt}[1]{\widetilde{#1}}
\newcommand\numberthis{\addtocounter{equation}{1}\tag{\theequation}}

\def\R{\mathbb{R}}
\def\N{\mathbb{N}}
\def\P{\mathcal{P}}
\def\Pac{\P_{\mathrm{ac},2}}
\def\A{\mathcal{A}}

\def\e{\varepsilon}

\def\grad{\nabla}
\def\div{\nabla\cdot}
\def\d{\,\mathrm{d}}
\def\p{\partial}
\def\ird{\int_{\R^d}}
\def\irdrd{\int_{\R^d\times\R^d}}

\def\:{\colon}
\def\der{\mathrm{d}}
\def\Ep{E_N^p}
\def\E{E_N}

\def\wtE{\wt{E}_N}
\def\bx{\boldsymbol{x}}
\def\by{\boldsymbol{y}}

\DeclareMathOperator{\minp}{min}
\DeclareMathOperator{\argmin}{argmin}
\DeclareMathOperator{\erf}{erf}

\DeclareMathOperator{\graph}{graph}

\begin{document}

\title{Numerical Study of a Particle Method for Gradient Flows}

\author{J. A. Carrillo}
\address{Department of Mathematics, Imperial College London, South Kensington Campus, London SW7 2AZ, UK.}
\email{carrillo@imperial.ac.uk}
\author{Y. Huang}
\address{School of Mathematics, University of Manchester, Oxford Road, Manchester M13 9PL, UK.}
\email{yanghong.huang@manchester.ac.uk}
\author{F. S. Patacchini}
\address{Department of Mathematics, Imperial College London, South Kensington Campus, London SW7 2AZ, UK.}
\email{f.patacchini13@imperial.ac.uk}
\author{G. Wolansky}
\address{Mathematics Dept., Technion--Israel Institute of Technology, Haifa 32000, Israel.}
\email{gershonw@math.technion.ac.il}

\keywords{Particle method, diffusion, aggregation, gradient flow, discrete gradient flow, JKO scheme.}
\subjclass[2010]{65M12, 35K05}
\date{6 December 2016}

\maketitle

\begin{abstract}
	We study the numerical behaviour of a particle method for gradient flows involving linear and nonlinear diffusion. This method relies on the discretisation of the energy via non-overlapping balls centred at the particles. The resulting scheme preserves the gradient flow structure at the particle level and enables us to obtain a gradient descent formulation after time discretisation. We give several simulations to illustrate the validity of this method, as well as a detailed study of one-dimensional aggregation-diffusion equations.
\end{abstract}


\section{Introduction}
\label{sec:introduction}

In this work we introduce a new particle method for approximating the solutions to evolution equations of the form
\begin{equation}\label{1.1}
	\begin{cases} \rho_t=\div\big[\rho\nabla\big(H'(\rho(x))+V(x)+
 W\ast\rho(x)\big)\big],  \quad t>0,\, x\in\R^d,\\
	\rho(0,\cdot)=\rho_0(\cdot),
	\end{cases}
\end{equation}
where $\rho(t,\cdot)\ge0$ is the unknown probability measure and $\rho_0$ is a fixed element of $\P_2(\R^d)$, the set of Borel probability measures on $\R^d$ with bounded second moment. Note that we denote by the same symbol a probability measure and its density, whenever the latter exists. The operator $*$ denotes the convolution, $H \: [0,\infty) \to \R$ is the \emph{density of internal energy}, $V\: \R^d \to \R$ is the \emph{confinement potential}, and $W\: \R^d \to \R$ is the \emph{interaction potential}. These equations are ubiquitous in many applications, ranging from granular media and porous medium flows to collective behaviour models in mathematical biology and self-assembly, see \cite{BCCP,Blanchet,CCH2,TBL,Vazquez} and the references therein.

Recent advances in the analysis of the equation~\eqref{1.1} are mainly based on variational schemes using the natural gradient flow structure in the space of probability measures, see e.g. \cite{McCann,JKO,Otto,Villani,AGS,cmcv-03}. We define the \emph{continuum energy functional} $E \: \P_2(\R^d) \to \R\cup\{-\infty,+\infty\}$ by
\bes
	E(\rho) = \begin{cases}
		\displaystyle \ird \big[ H(\rho(x)) + V(x)\rho(x) + \textstyle{\frac 12} W*\rho(x)\rho(x) \big]\d x & \mbox{if $\rho \in \Pac(\R^d)$},\\[0.3cm]
		\displaystyle \ird \big[ V(x) + \textstyle{\frac 12} W*\rho(x) \big]\d \rho(x) &\mbox{$\let\scriptstyle\textstyle\substack{\mbox{if}\, \rho \not \in \Pac(\R^d)\\ \mbox{and}\, H = 0,}$}\\[0.3cm]
		+\infty & \mbox{otherwise},
	\end{cases}
\ees
where $\Pac(\R^d)$ is the subset of $\P_2(\R^d)$ of probability measures which are absolutely continuous with respect to the Lebesgue measure. The functions $H, V$ and $W$ satisfy the following hypotheses.
\begin{hyp} \label{hyp:VW}
	$V$ is a function in $C^1(\R^d)$ and $W$ is a symmetric, locally integrable function in $C^1(\R^d\setminus\{0\})$ with $W(0) = 0$.
\end{hyp}
\begin{hyp} \label{hyp:basic}
	$H$ is a convex function in $C^1((0,\infty)) \cap C^0([0,\infty))$ with superlinear growth at infinity and $H(0) = 0$. Furthermore, $h( \lambda) := \lambda^d H\left(\lambda^{-d}\right)$ is convex and non-increasing on $(0,\infty)$.
\end{hyp}
The assumption $W(0) = 0$ is made without loss of generality. Indeed, if $W(0)$ is finite, then $W$ can be shifted ``up'' or ``down'' to get $W(0) = 0$; if $W$ has a singularity at $0$, then setting $W(0) := 0$ does not affect the physical behaviour of a system governed by the potential $W$. The assumptions that $H(0) = 0$ and $h$ is convex and non-increasing imply that the energy $E$ is displacement convex if, for example, $V=W=0$, see \cite{McCann}, \cite[Section 4]{McCann2} and \cite[Theorem 5.15]{Villani}. Also note that the classical cases $H(\rho) =  \rho\log \rho$ and $H(\rho) =  \frac{\rho^m}{m-1}$ ($m>1$) satisfy all the required assumptions.

The underlying topology on the probability measures in this paper is given by the \emph{quadratic Wasserstein distance} $d_2(\rho,\mu)$,
 which is defined between two measures $\rho$ and $\mu$ in $\P_2(\R^d)$ by
\[
d_2(\rho,\mu) = \inf_{\gamma \in \Pi(\rho,\mu)} \left[ \irdrd |x-y|^2 \d\gamma(x,y)\right]^{\frac12},
\]
where $\Pi(\rho,\mu)$ is the space of probability measures (also called transport plans) on $\mathbb{R}^d\times \mathbb{R}^d$ with first marginal $\rho$ and second marginal $\mu$. Let us fix a final time $T>0$. We say that $\rho\: [0,T] \to \P_2(\R^d)$ is a \emph{continuum gradient flow solution} with initial condition $\rho_0$ if
\be \label{eq:gradient-flow}
	\begin{cases}
		\rho'(t) = - \grad_{\P_{2}(\R^d)} E(\rho(t)),\\
		\rho(0) = \rho_0,
	\end{cases}
\ee
holds in the sense of distributions on $[0,T]\times\R^d$, see \cite[Equation (8.3.8)]{AGS}. The operator $\grad_{\P_{2}(\R^d)}$ is the classical \emph{quadratic Wasserstein gradient} on $\P_{2}(\R^d)$, which takes the explicit form
\bes
	\grad_{\P_{2}(\R^d)} E(\rho) = -\div \left(\rho \grad \dfrac{\delta E}{\delta \rho}\right) \quad \mbox{for all $\rho \in \P_{2}(\R^d)$},
\ees
where $\frac{\delta E}{\delta \rho} = H'(\rho) + V + W\ast\rho$ is the first variation density of $E$ at point $\rho$. As a by-product of the general theory developed for instance in \cite{AGS}, gradient flow solutions to \eqref{eq:gradient-flow} are weak solutions to \eqref{1.1}. Note that if theoretical issues such as the existence and uniqueness of solutions to the continuum gradient flow \eqref{eq:gradient-flow} are of interest, appropriate additional assumptions must be imposed on $H,V,W$ and $\rho_0$, see \cite{JKO,Villani,AGS}.

We propose below to approximate solutions to the continuum gradient flow \eqref{eq:gradient-flow} by finite atomic probability measures represented by finite numbers of particles. The basic idea is to restrict the continuum gradient flow to the discrete setting of atomic measures by performing the steepest descent of a suitable approximation of the energy $E$ defined on finite numbers of Dirac masses, and apply a discrete analogue of the JKO scheme for the Fokker-Planck equation proposed by Jordan, Kinderlehrer and Otto~\cite{JKO}. The theoretical underpinning of this method is studied in the companion paper~\cite{CPSW}, where the convergence of the discrete gradient flow to the continuum one is proved in the framework proposed by Serfaty in~\cite{Serfaty,SS}, in the special case of $V=W=0$ in one dimension for equally-weighted particles and under additional appropriate hypotheses on $H$. The goal of the present paper is to give numerical evidence of such convergence and to motivate possible extensions of the theoretical result in \cite{CPSW} to nonzero confinement and interaction potentials (even with possible singularities), as well as to higher dimensions and unequally-weighted particles.

Let us mention that other numerical methods have been developed to conserve particular properties of solutions of the continuity equation \eqref{1.1}. In \cite{Filbet,CCH2} the authors developed finite-volume methods preserving the decay of energy at the semi-discrete level, along with other important properties like non-negativity and mass conservation. Particle methods for these equations without the diffusion term are known to be convergent under suitable assumptions on the potentials $V$ and $W$ since these results are connected to the question of the mean-field limit and rigorous derivation of the equations from particle trajectories, see \cite{CCH} for instance. In the case of diffusion equations, particle methods based on suitable regularisations of the flux of the continuity equation \eqref{1.1} have been proposed in \cite{Russo,DM,LM,MGallic}; note that an early particle method was derived for collisional kinetic equations in \cite{Russo2}. Our steepest descent method is purely variational and is based on regularising the internal part of the energy $E$ by substituting particles by non-overlapping blobs, as we discuss next. Let us  mention that the numerical approximation of the JKO variational scheme has already been tackled by different methods using pseudo-inverse distributions in one dimension \cite{GT,Blanchet,CG}, diffeomorphisms \cite{CM}, or solving for the optimal map in a JKO step \cite{BCMO}. Our method avoids these computationally intensive procedures by the approximation of the energy in the discrete setting. Finally, note that gradient-flow-based Lagrangian methods for higher-order, drift diffusion and Fokker-Planck equations have recently been proposed in \cite{MO1,MO2,MO3,JMO}.

The paper is structured as follows. In Section \ref{sec:particle-method} we give a summary of the method and the derivation of the discrete gradient flow in a more general setting than that considered in~\cite{CPSW}. Section \ref{sec:jko-scheme} is dedicated to the derivation of the numerical scheme used to approximate the continuum gradient flow \eqref{eq:gradient-flow}, that is an explicit version of the JKO scheme, as well as to a numerical validation study of the scheme via diffusion equations. We stress that, from the scientific computing point of view, this is a preliminary study whose aim is to motivate the presented particle method. Finally, in Section \ref{sec:aggregation-diffusion} we give the numerical results for various one-dimensional aggregation-diffusion equations---we emphasise that this particle method, despite its simplicity, is able to capture the critical mass for the modified one-dimensional Keller-Segel model.

\section{Particle method and discrete gradient flows}
\label{sec:particle-method}

In this method, the underlying probability measure is characterised by the particles' positions $(x_1,\dots,x_N)\in (\R^d)^N = \R^{Nd}$ and the associated weights (or masses) $w:=(w_1,\cdots,w_N) \in \mathbb{R}^N$, where $N$ is the total number of particles considered. Throughout this paper, the positions $(x_1,\dots,x_N)$ are evolving in time but the weights $w$ are fixed and such that $w_i>0$ and $\sum_{i=1}^N w_i=1$. Also, we denote by $\R_w^{Nd}$ the space of particles with weights $w$, that is, $\bx:= (x_1,\dots,x_N) \in \R_w^{Nd}$ means that each particle $x_i$ is in $\R^d$ and is associated with the weight $w_i$. Notice the boldface font when referring to elements of $\R_w^{Nd}$.

\begin{rem}\label{rem:convention}
As an important convention in the rest of the paper, whenever particles $\bx\in \R_w^{Nd}$ are considered, they are assumed to be distinct, i.e., $x_i \neq x_j$ if $i\neq j$. Moreover, in one dimension, the particles are assumed to be sorted increasingly, i.e., $x_{i+1} >  x_i$ for all $i \in \{1,\dots,N-1\}$.
\end{rem}

\subsection{Discrete gradient flow}

Consider $N$ particles $\bx:=(x_{1},\dots,x_{N}) \in \R_w^{Nd}$ with fixed weights $w$. The most natural representation of the underlying probability measure is the empirical measure
\be \label{eq:empirical}
	\bx \mapsto \mu_N = \sum_{i=1}^N w_{i} \delta_{x_{i}},
\ee
which belongs to the space of \emph{atomic measures}
\bes
	\A_{N,w}(\R^d) := \left\{\mu \in \P_2(\R^d) \mid \exists\, \bx \in \R_w^{Nd}, \mu = \sum_{i = 1}^N w_{i}\delta_{x_{i}} \right\}.
\ees
\begin{defn}[Discrete energy] \label{defn:discrete-energy}
We define the \emph{discrete energy} $\E \: \A_{N,w}(\R^d) \to \R$ by
\be \label{eq:energy-discrete}
	E_N(\mu_N) = \sum_{i=1}^N |B_i| H\left(\frac{w_{i}}{|B_i|}\right) + \sum_{i = 1}^N w_{i} V(x_i) + \dfrac{1}{2} \sum_{i=1}^N \sum_{\substack{j=1\\j\neq i}}^N w_{i}w_{j} W(x_i - x_j),
\ee
where $B_i$ is the open ball of centre $x_i$ and radius $\frac{1}{2}\min_{j\neq i}|x_i - x_j|$, and $|B_i|$ is the volume of $B_i$.
\end{defn}
Note that $\E$ is finite over the whole $\A_{N,w}(\R^d)$ since, by the Hypotheses \ref{hyp:VW} and \ref{hyp:basic}, $H, V$ and $W$ are pointwise finite. The essence of this discrete approximation $E_N$ on $\A_{N,w}(\R^d)$ of the continuum energy $E$ on $\mathcal{P}_2(\mathbb{R}^d)$  lies in the treatment of the internal part $\int_{\mathbb{R}^d} H\big(\rho(x)\big)\d x$ of the energy $E$, which becomes infinity on point masses;  the point mass of each particle is uniformly spread to circumvent this problem. To this end, consider
\be \label{eq:density-on-balls}
	\bx \mapsto \rho_N = \sum_{i = 1}^N w_{i} \frac{\chi_{B_{i}}}{|B_{i}|},
\ee
where $\chi_{B_{i}}$ is the characteristic function of the ball $B_{i}$. Clearly $\rho_N$ is in $\Pac(\R^d)$, and thus the internal part of the energy  is well-defined for $\rho_N$. Note that the representation $\rho_N$ does not involve overlapping of balls, but contains ``gaps" between balls whose sizes are expected to decrease as the number of particles increases. Then, using \eqref{eq:density-on-balls} for the internal part of the energy and \eqref{eq:empirical} for the confinement and interaction parts, $E_N$ defined in \eqref{eq:energy-discrete} is exactly
\bes
	\ird H(\rho_N(x)) \d x + \ird V(x) \d \mu_N(x) + \irdrd W(x-y) \d\mu_N(y) \d\mu_N(x).
\ees
Here the diagonal terms in the interaction potential vanish since $W(0) = 0$ by Hypothesis \ref{hyp:VW}. This choice of non-overlapping particles has the main advantage of reducing the computational cost of the internal part of the discrete energy functional. One can already notice that this approach allows us to treat diffusive effects simultaneously with confinement and interaction ones in a very natural way; this is another advantage of our method, as it becomes clearer throughout this paper.

Since the expression above depends essentially on $\bx \in \R_w^{Nd}$, we can define the discrete energy equivalently as a function of $\bx \in \R_w^{Nd}$
instead of $\mu_N \in \A_{N,w}(\R^d)$:
\be \label{eq:energy-discrete-particles}
	\wt{E}_N(\bx) := \E(\mu_N) \quad \mbox{for all $\mu_N \in \A_{N,w}(\R^d)$ with particles $\bx \in \R_w^{Nd}$}.
\ee

\begin{defn}[Subdifferential in Hilbert spaces] \label{defn:subdifferential}
	Let $X$ be a Hilbert space with inner product $\langle\cdot,\cdot\rangle_X$ and $\phi\: X \to \R$. We define the \emph{subdifferential} $\p \phi\: X \to 2^X$ of $\phi$, for all $x\in X$, by the subset
\bes
	\p \phi(x) = \left\{z\in X \mid \liminf_{y\to x} \frac{\phi(y) - \phi(x) - \langle z,y-x\rangle_X}{|y-x|_X} \geq 0\right\}.
\ees
If $\partial \phi(x)$ is not empty, then $\partial^0 \phi(x)$ is defined to be the unique element in $\partial \phi(x)$ with minimal norm.
\end{defn}

Now, we say that $\bx\: [0,T] \to \R_w^{Nd}$ is a \emph{discrete gradient flow solution} with initial condition $\boldsymbol{x^0} \in \R^{Nd}$ if the differential inclusion
\be \label{eq:discrete-gradient-flow-inclusion0}
		w\cdot \bx'(t) \in - \p \wtE(\bx(t))\qquad  \mbox{for almost all $t \in (0,T]$}
\ee
is satisfied, and $\bx(0)=\boldsymbol{x^0}$. Here $\bx'(t)$ is the velocity of the curve $\bx(t)$ and $w\cdot \bx'$ is the element-wise product $(w_{1}x'_{1},\dots,w_{N}x'_{N})$. This formulation is not  a standard differential inclusion because of the presence of the weights. To cope with this, we introduce the following inner product on $\R_w^{Nd}$.
\begin{defn}[Weighted inner product on $\R_w^{Nd}$] \label{defn:weighted-ip}
	For all $\bx,\by\in\R_w^{Nd}$ we define the \emph{weighted inner product} between $\bx$ and $\by$ as
\bes
	\langle \bx,\by \rangle_w = \sum_{i=1}^N w_{i}\langle x_i,y_i\rangle_{\R^d}.
\ees
\end{defn}
From now on, the Euclidean space $\R_w^{Nd}$ is endowed with this inner product. This definition clearly induces the weighted norm
\be \label{eq:norm}
	|\bx|_w := \sqrt{\langle \bx,\bx \rangle_w} = \sqrt{\textstyle{\sum}_{i=1}^N w_{i}|x_i|_{\R^d}^2} \quad \mbox{for all $\bx\in\R_w^{Nd}$}.
\ee
It also induces a slightly more general subdifferential structure: for any functional $\phi\:\R_w^{Nd} \to \R$,
\bes
	\p_w \phi(\bx) := \left\{\boldsymbol{z}\in\R_w^{Nd} \mid w\cdot \boldsymbol{z} \in \p \phi(\bx)\right\} \quad \mbox{for all $\bx\in\R_w^{Nd}$},
\ees
where again $w\cdot \boldsymbol{z}$ is the element-wise product between $w$ and $\boldsymbol{z}$; and we can then define the element $\partial_w^0 \phi(\bx)$ with minimal norm accordingly. The discrete gradient flow inclusion \eqref{eq:discrete-gradient-flow-inclusion0} can now be rewritten more naturally as
\be \label{eq:discrete-gradient-flow-inclusion}
	\begin{cases}
		\bx'(t) \in - \p_w \wtE(\bx(t)) & \mbox{for almost all $t \in (0,T]$},\\
		\bx(0) = \boldsymbol{x^0}.
	\end{cases}
\ee
This is the discrete gradient flow structure that is used from now on.

\begin{rem}
In~\cite{CPSW} the authors showed that in one dimension with $V=W=0$ the discrete energy $\E$ $\Gamma$-converges (in the $d_2$-topology, see Definition \ref{defn:gamma-convergence}) to the continuum one $E$ if adequate additional assumptions are imposed on $H$. 
\end{rem}

\subsection{$p$-approximated discrete gradient flow}
\label{subsec:approximated-gradient-flow}

The gradient flow~\eqref{eq:discrete-gradient-flow-inclusion} is written as a differential inclusion instead of an ordinary differential equation, primarily because the radius $\frac{1}{2}\min_{j\neq i} |x_i-x_j|$ of the balls $B_i$ is not a smooth function of the positions in the definition of the energy $\wtE$. As a result, the energy $\tilde{E}_N$ is not a smooth function of the positions $\bx$ when the radius of the ball $B_i$ is determined by more than one particle.

If $\wtE$ is lower semi-continuous and convex, and $\boldsymbol{x^0}\in  \{\by\in \R_w^{Nd} \mid \p_w\wtE(\by) \neq \emptyset\}$, it is known by~\cite[Theorem 1 of Section 3.2]{Aubin} that the discrete gradient flow~\eqref{eq:discrete-gradient-flow-inclusion} is well-posed and the equation $\bx'(t) = -\p_w^0\wtE(\bx(t))$ is satisfied for almost every $t\in [0,T]$. Note that, in our case, the lower semi-continuity of $\wtE$ is trivial to check by the assumptions on $H$, whereas the convexity is proved in Proposition \ref{prop:convexity-EN} for convex potentials $V$ and $W$ in one dimension. Although the element of minimal norm $\partial_w^0 \wtE(\bx(t))$ can be computed in some situations as in~\cite{CPSW} (for the special case of equally-weighted particles with internal energy only and in one dimension), the detailed procedure
is usually involved and is not convenient for numerical implementation. In theory,
one can approximate $\wtE$ by a differentiable function usually characterised by a parmeter, where the minimal norm element of the subdifferential is recovered as the limit 
of the gradient of the approximation as the parameter goes to infinity. A common choice of such an approximation is
\bes
	Y_N^p(\bx) := \inf_{\by\in\R_w^{Nd}} \left(\wtE(\by) + \textstyle{\frac{p}{2}}|\bx-\by|_w^2 \right),
\ees
depending on the parameter $p>0$. Then $Y_N^p$ is differentiable almost everywhere in $\R_w^{Nd}$, $\grad Y_N^p$ is called the \emph{Yosida approximation} of $\p_w \wtE$, and
\bes
	\grad Y_N^p(\bx) \xrightarrow[p \to \infty]{} \p_w^0 \wtE(\bx) \quad \mbox{for all $\bx \in \R_w^{Nd}$},
\ees
see \cite[Theorem 2 of Section 3.1 and Theorem 4 of Section 3.4]{Aubin}. However, the Yosida approximation requires another optimisation, and is therefore not very well adapted to the computation of $\p_w^0 \wtE(\bx)$. An alternative is to approximate the radius of the balls $B_i$ in the definition of $\wtE$ by a smooth function.

\begin{defn}[$p$-approximation of the minimum function] \label{defn:general-p-min}
For $p > 0$ and integer $s \geq 2$, the function $\minp_p \colon (0,\infty)^s \to (0,\infty)$
is defined as
\bes
\minp_p(x_1,\dots,x_s) =
\begin{cases}
	\left(\dfrac 1s  \displaystyle \sum_{i=1}^s x_i^{-p}\right)^{-\frac1p} &
	\mbox{if $x_1,\dots,x_s \in (0,\infty)$},\\
	0 & \mbox{otherwise.}
\end{cases}
\ees
We call $\minp_p(x_1,\dots,x_s)$ the \emph{$p$-approximation} of the minimum function
$\min (x_1,\cdots,x_s)$.
\end{defn}
The following proposition, whose proof is straightforward and left to the reader, justifies the use of the $p$-approximation of the minimum function.
\begin{prop} \label{prop:min-p}
	Let $s \geq 2$. The following statements hold.
\begin{enumerate}[label=(\arabic*)]
	\item \label{it:cinf} Let $p>0$. Then $\minp_p \in C^\infty((0,\infty)^s)$.\\[-0.3cm]
	\item  \label{it:lim} $\minp_p(x) \to \min(x)$ as $p\to\infty$ for all $x \in (0,\infty)^s$.\\[-0.3cm]
	\item  \label{it:noninc} $p\mapsto \minp_p(x)$ is non-increasing on $(0,\infty)$ for all $x \in (0,\infty)^s$.\\[-0.3cm]
	\item  \label{it:below} Let $p>0$. Then $\minp_p(x) \geq \min(x)$ for all $x \in (0,\infty)^s$, with equality if $x_i=x_j$ for all $i,j$.\\[-0.3cm]
	\item  \label{it:above} Let $p \geq 1$. Then $\minp_p(x) \leq s^{1/p}\min(x)$ for all $x \in (0,\infty)^s$.
\end{enumerate}
\end{prop}

Now the discrete energy is further approximated using a smoothed radius as below.
\begin{defn}[$p$-approximated discrete energy] \label{defn:p-discrete-energy}
Let $p>0$. The \emph{$p$-approximated discrete energy}
$\Ep \: \A_{N,w}(\R^d) \to \R$ is defined by
\be \label{eq:general-energy-discrete-p}
	\Ep(\mu_N) = \sum_{i=1}^N |B_i^p| H\left(\frac{w_{i}}{|B_i^p|}\right) + \sum_{i = 1}^N w_{i} V(x_i) + \dfrac{1}{2} \sum_{i=1}^N \sum_{\substack{j=1\\j\neq i}}^N w_{i}w_{j} W(x_i - x_j),
\ee
for all $\mu_N\in\A_{N,w}(\R^d)$ with particles $\bx\in\R_w^{Nd}$. Here $B_i^p$ is the ball with the new radius
\bes
 \textstyle{\frac{1}{2}\minp_p\{|x_i-x_j| \mid j\neq i\}
 = \frac{1}{2}\left(\frac{1}{N-1}\sum_{\substack{j=1 \\ j\neq i}}^N |x_i-x_j|^{-p}\right)^{-\frac1p}}.
\ees
\end{defn}

As in \eqref{eq:energy-discrete-particles}, we can define the $p$-approximated discrete energy on $\R_w^{Nd}$ rather than on $\A_{N,w}(\R^d)$:
\be \label{eq:energy-discrete-particles-p}
	\wtE^p(\bx) := \Ep(\mu_N) \quad \mbox{for all $\mu_N \in \A_{N,w}(\R^d)$ with particles $\bx \in \R_w^{Nd}$}.
\ee
We say that $\bx\:[0,T] \to \R_w^{Nd}$ is a \emph{p-approximated discrete gradient flow solution} with initial condition $\boldsymbol{x^{0}} \in \R_w^{Nd}$ if
\be \label{eq:discrete-gradient-flow-ode}
	\begin{cases}
		\bx'(t) = - \grad_w \wtE^p(\bx(t)) & \mbox{for almost all $t \in (0,T]$},\\
		\bx(0) = \boldsymbol{x^{0}},
	\end{cases}
\ee
where the weighted gradient $\grad_w$ is naturally defined on $\R_w^{Nd}$ by
\bes
	\grad_w := \left(\frac{1}{w_{1}}\frac{\p}{\p x_1},\dots,\frac{1}{w_{N}}\frac{\p}{\p x_N}\right).
\ees

\begin{rem}
	Whenever $d=1$ and $V$ and $W$ are convex, we know by Proposition \ref{prop:convexity-EN} that $\wtE^p$ is convex, and therefore the $p$-approximated discrete gradient flow \eqref{eq:discrete-gradient-flow-ode} is well-posed.
\end{rem}

\begin{rem}
	The interest of the $p$-approximation of the discrete gradient flow only lies in the numerical simplicity of coding a gradient descent on an ODE system such as \eqref{eq:discrete-gradient-flow-ode} rather than on a differential inclusion such as \eqref{eq:discrete-gradient-flow-inclusion}. The $p$-approximation is indeed not needed for the theoretical proof given in \cite{CPSW} of the convergence of the discrete gradient flow to the continuum one, since, as already mentioned, in that specific case the minimal norm element of $\p\wtE$ is explicitly computable.
\end{rem}
\begin{rem}
Approximations of the minimum function other than the one given in Definition \ref{defn:general-p-min} are possible. One example is
	\bes
		\minp_p(x_1,\dots,x_s) =
		\begin{cases}
			-\frac 1p \log\left(\frac 1s \sum_{i=1}^s e^{-px_i}\right) & \mbox{if $x_1,\dots,x_s \in (0,\infty)$},\\
			0 & \mbox{otherwise}.
		\end{cases}
	\ees
\end{rem}
The properties of $\minp_p$ given in Proposition \ref{prop:min-p} are not enough to say that the $p$-approximated discrete gradient flow \eqref{eq:discrete-gradient-flow-ode} converges to the discrete one \eqref{eq:discrete-gradient-flow-inclusion} as $p\to\infty$ in some sense. This still needs to be checked if we want to justify the numerical use of the $p$-approximated gradient flow. In the next section and in Appendix \ref{app:convergence} we prove that this is the case in dimension one since we can there exploit the convexity of the discrete energies.

\subsection{One-dimensional case}
\label{subsec:one-d}

\begin{defn}[Inter-particle distance] \label{defn:interparticles}
For any particles $\bx \in \R_w^N$ we denote the \emph{inter-particle distance} by the positive quantity (eventually $+\infty$ by convention)
\bes
	\Delta x_i := x_i-x_{i-1} \quad \mbox{for $i \in \{1,\dots,N+1\}$},
\ees
with the convention $x_0 = -\infty$ and $x_{N+1} = +\infty$. We also write, for $p>0$,
\bes
	r_i = \min(\Delta x_i,\Delta x_{i+1}) \quad \mbox{and} \quad r_i^p = \minp_p(\Delta x_i,\Delta x_{i+1}) \quad \mbox{for all $i \in \{1,\dots,N\}$}.
\ees
Note that $r_i = |B_i|$ and $r_i^p = |B_i^p|$.
\end{defn}
Since particles are sorted increasingly by convention, the $p$-approximated discrete energy \eqref{eq:general-energy-discrete-p} in dimension one can be defined in the simpler form
\be \label{eq:energy-p-one-d}
	\Ep(\mu_N) = \sum_{i=1}^N r_i^p H\left(\frac{w_i}{r_i^p}\right) + \sum_{i = 1}^N w_i V(x_i) + \dfrac{1}{2} \sum_{i=1}^N \sum_{\substack{j=1\\j\neq i}}^N w_iw_j W(x_i - x_j),
\ee
for all $\mu_N\in\A_{N,w}(\R)$ with particles $\bx\in\R_w^{Nd}$. The equivalent formulation of the energy $\wtE^p$ on the particles is defined in dimension one accordingly, see \eqref{eq:energy-discrete-particles-p}.

In order to check that the $p$-approximated discrete gradient flow defined above indeed approximates the discrete one, we need a few elements of maximal monotone operator theory (see \cite{Attouch,Aubin,Brezis} for a detailed overview of the theory). These notions and results being quite abstract, they are only given in Appendix \ref{app:convergence}. There we show how to use this theory to prove the convergence of the $p$-approximated discrete gradient flow to the discrete one in a precise sense, in the case when $V$ and $W$ are convex; we therefore justify the numerical use of the $p$-approximated gradient flow \eqref{eq:discrete-gradient-flow-ode} (at least in the case when $V$ and $W$ are convex and $d=1$).

\begin{rem}
In this method, one reason why we decide to discretise according to non-overlapping balls, rather than, for example, Voronoi cells (see for instance \cite{Graf} for a detailed account on Voronoi cells), is to allow for simpler computations of the discrete energies, which in turn gives rise to less costly simulations. In fact, this is not very relevant in one dimension since then the implementation of Voronoi cells is actually not more costly than that of non-overlapping balls; the one-dimensional simulations given below, which are run according to the non-overlapping balls described above, should therefore be seen as an initial validation of the method and as test cases which need to be extended to higher dimensions in a future work.
\end{rem}

\section{Numerical scheme and validation}
\label{sec:jko-scheme}

Here a general variational formulation is applied to the discrete setting described above, leading to our numerical scheme giving the particles' positions at discrete time steps (see \cite{AGS,Ambrosio,JKO} for a Fokker-Planck motivation and its generalisation to curves in probability spaces). For the sake of generality the derivation of the scheme is partly done for the discrete gradient flow, rather than the $p$-approximated one. The simulations given later were, however, performed in the $p$-approximated setting only. Note also that, in this section, some indices $N$ and $p$ are dropped for clarity.

\subsection{The scheme}
\label{subsec:jko}

Take $(t_n)_{n=0}^{M} \subset [0,T]$ a subdivision of the time interval $[0,T]$ and a time-step size $\Delta t=\Delta_n t$ (eventually adaptive) such that $t_{n} = t_0 + \sum_{i=0}^{n-1}\Delta_i t$ for all $n\in\{1,\dots,M\}$. Suppose that, for some $n \in \{0,\dots,M-1\}$, we know $\mu^n \in \A_{N,w}(\R^d)$, where $\mu^n$ is the empirical measure of an approximated discrete gradient flow solution at time $t_n$. Then we want to get an approximated discrete gradient flow solution at the time $t_{n+1}$, that is $\mu^{n+1}$. To this end we use the JKO scheme, i.e., we choose $\mu^{n+1}$ as
\be \label{eq:general-jko}
	\mu^{n+1} := \underset{\mu \in \A_{N,w}(\R^d)}{\argmin} \left\{ \frac{1}{2\Delta t} d_2^2(\mu^n,\mu) + \E(\mu)\right\},
\ee
where we recall that $d_2$ is the quadratic Wasserstein distance. In the following any object with sub- or superscript $n$ or $n+1$ is associated with the time step $n$ or $n+1$, respectively.

From now on, apart from a two-dimensional test in Section \ref{subsubsec:two-dimensions}, the setting is \emph{one-dimensional}.

\subsubsection{Computation of the Wasserstein distance}
\label{subsubsec:wasserstein}
Let us compute the Wasserstein distance in \eqref{eq:general-jko}. For each $n \in \{0,\dots,M-1\}$ and approximation $\mu^n$ and $\mu \in \A_{N,w}(\R)$, we directly get
\be \label{eq:distance}
	d_2^2(\mu^n,\mu) = \sum_{i = 1}^N w_{i} (x_i^n - x_i)^2,
\ee
where $\bx^n=(x_1^n,\dots,x_N^n)$ and $\bx=(x_1,\dots,x_N)$ are the particles of $\mu^n$ and $\mu$, respectively, which leads to the scheme
\be \label{eq:jko}
	\mu^{n+1} := \underset{\mu \in \A_{N,w}(\R)}{\argmin} \left(\sum_{i = 1}^N w_i \dfrac{(x_i^n - x_i)^2}{2\Delta t} + \E(\mu)\right).
\ee
Clearly the scheme \eqref{eq:jko} on the empirical measures can be equivalently rewritten on the particles:
\be \label{eq:jko-particles}
	w\cdot \bx^{n+1} := \underset{\bx \in \R_w^N}{\argmin} \left(\sum_{i = 1}^N w_i \dfrac{(x_i^n - x_{i})^2}{2\Delta t} + \wtE(\bx)\right),
\ee
where $w\cdot\bx^{n+1}$ is the element-wise multiplication between the vectors $w$ and $\bx^{n+1}$. The minimisation problem \eqref{eq:jko-particles} is characterised by
\be \label{eq:discrete-min-chara}
 0 \in \p_w\left( \sum_{i = 1}^N w_i \dfrac{(x_i^n - x^{n+1}_{i})^2}{2\Delta t} + \wtE(\bx^{n+1})\right).
\ee
When $\wtE$ is convex, \eqref{eq:discrete-min-chara} is precisely the backward Euler scheme of the differential inclusion~\eqref{eq:discrete-gradient-flow-inclusion}.

\subsubsection{Minimisation and final form of the scheme}
\label{subsubsec:minimisation}

Let $p > 0$. We now return to the $p$-approximated setting, i.e., we consider \eqref{eq:jko-particles} (and equivalently \eqref{eq:jko}) with $\wtE^p$ instead of $\wtE$ (and equivalently $\E^p$ instead of $\E$). We want to minimise, over the whole set $\R_w^N$, the functional in the $\argmin$ operator in \eqref{eq:jko-particles} (and \eqref{eq:jko}) to find our approximation $\bx^{n+1}$ (and $\mu^{n+1}$) at time step $n+1$. First note that, for each $n \in \{0,\dots,M-1\}$, we can write the discrete energy $\wtE^p$ as
\bes
	\wtE^p(\bx^{n+1}) = \sum_{i = 1}^N w_i E_i^{n+1},
\ees
where
\bes
	E_i^{n+1} := r_i^{n+1} H\left(\dfrac{w_i}{r_i^{n+1}}\right) + w_i V(x_i^{n+1}) + \dfrac{1}{2} \displaystyle \sum_{\substack{j=1\\j\neq i}}^N w_j W(x_i^{n+1} - x_j^{n+1}),
\ees
with $r_i^{n+1} := r_i^{p,n+1}$, see Definition \ref{defn:interparticles}. For any $n \in \{0,\dots,M-1\}$, \eqref{eq:discrete-min-chara} becomes
\be \label{eq:jko-final-implicit}
	x_i^{n+1} = x_i^n - \dfrac{\Delta t}{w_i} \sum_{j = 1}^N w_j\frac{\p E_j^{n+1}}{\p x_i^{n+1}},
\ee
where the uninspiring computation of the derivatives in the sum terms is left to the reader. The scheme \eqref{eq:jko-final-implicit} is the implicit Euler scheme of the ODE \eqref{eq:discrete-gradient-flow-ode}, or equivalently, of the ODE
\bes
 w_i \dfrac{\der x_i}{\der t}  = - \frac{\p \wtE^p}{\p x_i}(\bx), \quad i=1,2,\dots,N,
\ees
and it coincides with the JKO scheme for $\wtE^p$, as in \eqref{eq:general-jko}. Since the implicit scheme~\eqref{eq:jko-final-implicit} is difficult and costly to solve, its explicit Euler version was used in the numerical examples of this paper:
\be \label{eq:jko-final-explicit}
	x_i^{n+1} = x_i^n - \dfrac{\Delta t}{w_i} \sum_{j = 1}^N w_j\frac{\p E_j^{n}}{\p x_i^{n}}.
\ee
We remind the reader that the convergence of the implicit Euler version of the scheme is presented in \cite{JKO} in one dimension. The stability analysis of the explicit scheme \eqref{eq:jko-final-explicit} is not dealt with as it is not the purpose of the present paper---we are not worried about time stepping stability issues for large number of particles in this preliminary stage of validation of our new approach. We can nevertheless indicate that, intuitively, a time step satisfying $\Delta t \leq \frac{C}{N^2}$ for some constant $C>0$ should suffice to ensure stability; indeed, this would be in line with the classical CFL condition for diffusion equations where the mesh size is of order $\frac{1}{N}$. Note that this condition was respected in the simulations presented below. Let us also mention that higher-order time discretisations could be used in place of the explicit Euler scheme, and would indeed lead to a better time and space accuracy of the method.

\begin{rem} 
It is worth pointing out that our particle method does not aim at being competitive against classical finite-volume and finite-difference schemes for purely diffusive equations. In fact, it aims primarily at being simple to implement---even in higher dimensions---and flexible when considering additional terms to diffusion such as confinement and interaction. In terms of complexity at each time step, the method is of order $N^2$ regardless of the space dimension, as it actually is for other finite-volume methods \cite{CCH2}. Note that, however, if we time-discretise \eqref{eq:discrete-gradient-flow-inclusion} rather than its regularised form \eqref{eq:discrete-gradient-flow-ode}, the complexity becomes significantly higher in more than one dimension (at most of order $N^3$) since then we are required to find the closest neighbour to each particle; in dimension one the order of complexity is still $N^2$ thanks to the increasing ordering of the particles.
\end{rem}

\subsection{Initialisation of the scheme}
\label{subsec:initial-condition}

Below we give two different ways of approximating the initial profile. Let us first introduce the notion of pseudo-inverse.

\begin{defn}[Pseudo-inverse] \label{defn:generalised-inverse}
	Let $F\: \R \to [0,1]$ be a non-decreasing and right-continuous function. The \emph{pseudo-inverse} $\Phi \: [0,1] \to \R \cup \{-\infty,+\infty\}$ of $F$ is the non-decreasing and right-continuous function defined by
\bes
	\Phi(\e) = \inf \{x \in \R\mid F(x) > \e\} \quad \mbox{for all $\e \in [0,1]$}.
\ees
\end{defn}

\subsubsection{Initially equally-weighted particles}
\label{subsubsec:initially-equally-weighted-particles}

If we want to approximate the initial profile $\rho_0 \in \P_2(\R)$ with equally-weighted particles, we need to start with unequally-spaced particles. Consider the pseudo-inverse $\Phi_0$ of the cumulative distribution $x \mapsto \rho_0((-\infty,x])$ of the initial profile. Suppose $w_i = 1/N$ for all $i\in\{1,\dots,N\}$. Then choose the initial particles $\boldsymbol{x^0} := (x_1^0,\dots,x_N^0) \in \R_w^N$ as
\bes
	x_i^0 = \Phi_0\left(\dfrac{2i-1}{2N}\right) \quad \mbox{for all $1 \leq i \leq N$}.
\ees

\subsubsection{Initially equally-spaced particles}
\label{subsubsec:initially-equally-spaced-particles}

If now we want to approximate the initial profile $\rho_0 \in \P_2(\R)$ with equally-spaced particles, we need to assign a different weight to each particle. Consider $F_0$, the cumulative distribution of the initial profile, and $\boldsymbol{x^0} \in \R_w^N$, some chosen equally-spaced initial particles. Let us define
\bes
	\textstyle{x_{i\pm\frac12}^0 = \frac12\left(x_{i}^0 + x_{i\pm1}^0\right)} \quad \mbox{for all $1\leq i \leq N$}.
\ees
Then we choose the weights $w$ as
\be \label{eq:weights}
\begin{cases}
	w_1 = F_0\left(x_{1+\frac12}^0\right),\\[0.2cm]
	w_i = F_0\left(x_{i+\frac12}^0\right) - F_0\left(x_{i-\frac12}^0\right) \quad \mbox{for all $1 < i < N$},\\[0.2cm]
	w_N = 1 - F_0\left(x_{N-\frac12}^0\right).
\end{cases}
\ee
The weight $w_i$ as given in \eqref{eq:weights} for each $1 < i < N$ is the mass that $\rho_0$ ``assigns" to the interval $[x_{i-\frac12},x_{i+\frac12}]$, that is, to the Voronoi cell of $x_i$. Note that the choice we have on the initial positions of the particles $\boldsymbol{x^0}$ is not completely free of constraints: they must be chosen inside the support of $\rho_0$ or the resulting weights are zero. Finally, since the initial particles are chosen to be equally spaced, only the first and last particles' locations are needed to determine the locations of all the others; for the following numerical simulations we give this information under the form $I_\mt{init}=[x_1,x_N]$.

\begin{rem} \label{rem:advantages}
	There are two main advantages in the second initialisation approach. The first one is that it allows a better approximation of the initial profile when there is a strong variation in the density; indeed one has the freedom to place initial particles in less populated regions. The second one is that there are no gaps between the initial balls of centres $x_i^0$ and diameters $\minp_p(\Delta x_i^0,\Delta x_{i+1}^0) = \min(\Delta x_i^0,\Delta x_{i+1}^0)$, see Proposition \ref{prop:min-p}\ref{it:below}, allowing again a better approximation of the initial profile.
\end{rem}

\subsection{Computation of the error}
\label{subsec:computation-error}

The natural error to compute for our scheme \eqref{eq:jko-final-explicit} is the quadratic Wassertein error.

\subsubsection{Error with respect to an exact solution}
If we want to get the error between the numerical and exact solutions we proceed as follows. Let $\rho \in \P_2(\R)$ be a continuum gradient flow solution at the final time $T$. Also write $\bx\in\R_w^{N}$, the approximation of the $p$-approximated discrete gradient flow solution obtained with the JKO scheme \eqref{eq:jko-final-explicit} at the final time step $M$ with associated empirical measure $\mu := \mu^{M}$. Then we define the \emph{quadratic Wasserstein error} by
\bes
	e_{d_2} = d_2(\rho,\mu).
\ees
To compute this, one can use the one-dimensional pseudo-inverse definition of the quadratic Wasserstein distance. Let us write $F$ and $G$ the cumulative distributions of $\rho$ and $\mu$, respectively, and $\Phi$ and $\Psi$ their respective pseudo-inverses. Then
\bes
	e_{d_2} = \left(\int_0^1 \left(\Phi(\e) - \Psi(\e)\right)^2 \d\e\right)^{\frac12}.
\ees
We have, for all $i\in\{1,\dots,N\}$,
\[
	\Psi(\e) = x_i \quad \mbox{if }\ \e \in [\Omega_{i-1},\Omega_i),
\]
with $\Omega_i := \sum_{j=1}^{i} w_j$ and the convention $\Omega_0 = 0$. Therefore
\be \label{eq:error-wasserstein-1d-2}
	e_{d_2} = \left(\sum_{i=1}^N \int_{\Omega_{i-1}}^{\Omega_i} \left(x_i - \Phi(\e)\right)^2 \d\e\right)^{\frac12}.
\ee
The error can thus be easily determined if the inverse of the cumulative distribution $\Phi$ of the exact solution at the final time $T$ is known.

\subsubsection{Error with respect to a discrete steady state}
If we know that the considered gradient flow has a steady state and we are interested in the stabilisation behaviour of the scheme \eqref{eq:jko-final-explicit} we proceed as follows. Let $\boldsymbol{x^*}:= (x_1^*,\dots,x_N^*) \in \R_w^N$ be a discrete steady state (obtained by running a simulation for a ``large" time) with associated empirical measure $\mu^*$, and let $\bx\in\R_w^{N}$ be the approximation of the $p$-approximated discrete gradient flow solution obtained with the JKO scheme \eqref{eq:jko-final-explicit} at some final time step $M$ with empirical measure $\mu := \mu^{M}$. Then we define the error by
\be \label{eq:error-discrete}
	e_{d_2}^* = d_2(\mu^*,\mu) = \left(\sum_{i=1}^N w_i(x_i^* - x_i)^2\right)^{\frac12},
\ee
which is nothing but the \emph{Euclidean error} on the particle space $\R_w^N$, see \eqref{eq:norm}.

\subsection{Numerical validation: Diffusions}
\label{sec:results}

Before giving any simulation results, let us summarise practical implementation aspects of the scheme.

\begin{itemize}[leftmargin=*]
	\item As already justified in Section \ref{subsec:jko}, all the following simulations were obtained by implementing the explicit scheme \eqref{eq:jko-final-explicit}.
	\item The parameter $p$ needed for the iteration of such scheme was chosen to be 10 for every simulation. This choice is only justified by the fact that it gave sensible results. Note that taking $p$ ``large'' is a bad idea. Indeed, if $p$ increases, then the $p$-approximated discrete gradient flow ``approaches'' the discrete one, for which the energy $\wtE$ is not differentiable whenever two neighbouring inter-particle distances are equal. This may cause numerical instabilities as we would expect the inter-particle distances at the centre of a symmetric profile (like a Gaussian) to be indeed equal; this is what we observed for large $p$ under the form of particle oscillations.
	\item All one-dimensional simulations were initialised in the way described in Section \ref{subsubsec:initially-equally-spaced-particles}, up to a slight modification for end particles in Section \ref{subsec:compact}. In each case, the initial continuum profile and the interval $I_\mt{init}$ are explicitly given.
	\item The choice of the time-step size is explicitly given for each simulation.
	\item All the solution profiles given in the figures for the one-dimensional simulations were drawn by linking linearly the centres of every constant piece of the function $\rho_N$ defined in \eqref{eq:density-on-balls}. Note that this is only one way of representating the discrete solution and other choices are also possible.
\end{itemize}

In this section we validate our scheme \eqref{eq:jko-final-explicit} by showing test simulations run on diffusion equations: the heat and porous medium equations.

\subsubsection{The heat equation}
\label{subsec:heat-equation-1d}

The (linear) heat equation $\rho_t =\Delta \rho$ is a continuum gradient flow~\eqref{eq:gradient-flow} for $H(\rho)=\rho\log \rho$ and $V = W = 0$. The evolution of the solution starting from the initial data
\be\label{eq:rho0heat}
	\rho_0^\mt{heat}(x) = \frac{1}{\sqrt{4\pi t_0}}e^{-\frac{x^2}{4t_0}} \quad \mbox{with $t_0=0.25$}
\ee
is shown in Figure \ref{fig:heat-evo}. The exact solution is given, for all $t>0$, by
\bes
\rho(t,x) = \dfrac{1}{\sqrt{4\pi (t+t_0)}} e^{-\frac{x^2}{4(t+t_0)}} \quad \mbox{for all $x \in \R$},
\ees
whose cumulative distribution $F$ and its pseudo-inverse $\Phi$ are
\bes
\begin{cases}
F(t,x) = \dfrac{1}{2} \left(1 + \erf\left(\dfrac{x}{\sqrt{4(t+t_0)}}\right)\right) & \mbox{for all $x \in \R$},\\[0.4cm]
\Phi(t,\e) = \sqrt{4(t+t_0)} \erf^{-1}(2\e - 1) & \mbox{for all $\e \in [0,1)$},
\end{cases}
\ees
where $\erf$ is the error function. The error at the final time $T$ can then be found using \eqref{eq:error-wasserstein-1d-2} and some quadrature form to approximate the integrals therein, see Figure \ref{fig:heat-error}. In Figure \ref{fig:heat-evo} we chose $I_\mt{init}=[-2.5,2.5]$, whereas in Figure \ref{fig:heat-error} we chose $I_\mt{init} = [-4,4]$, see Section \ref{subsubsec:initially-equally-spaced-particles}.

\bfig[!ht]
\centering
	\subcaptionbox{Evolution for $N=50$ with $\Delta t = 10^{-5}$.\label{fig:heat-evo}}{\includegraphics[scale=0.41]{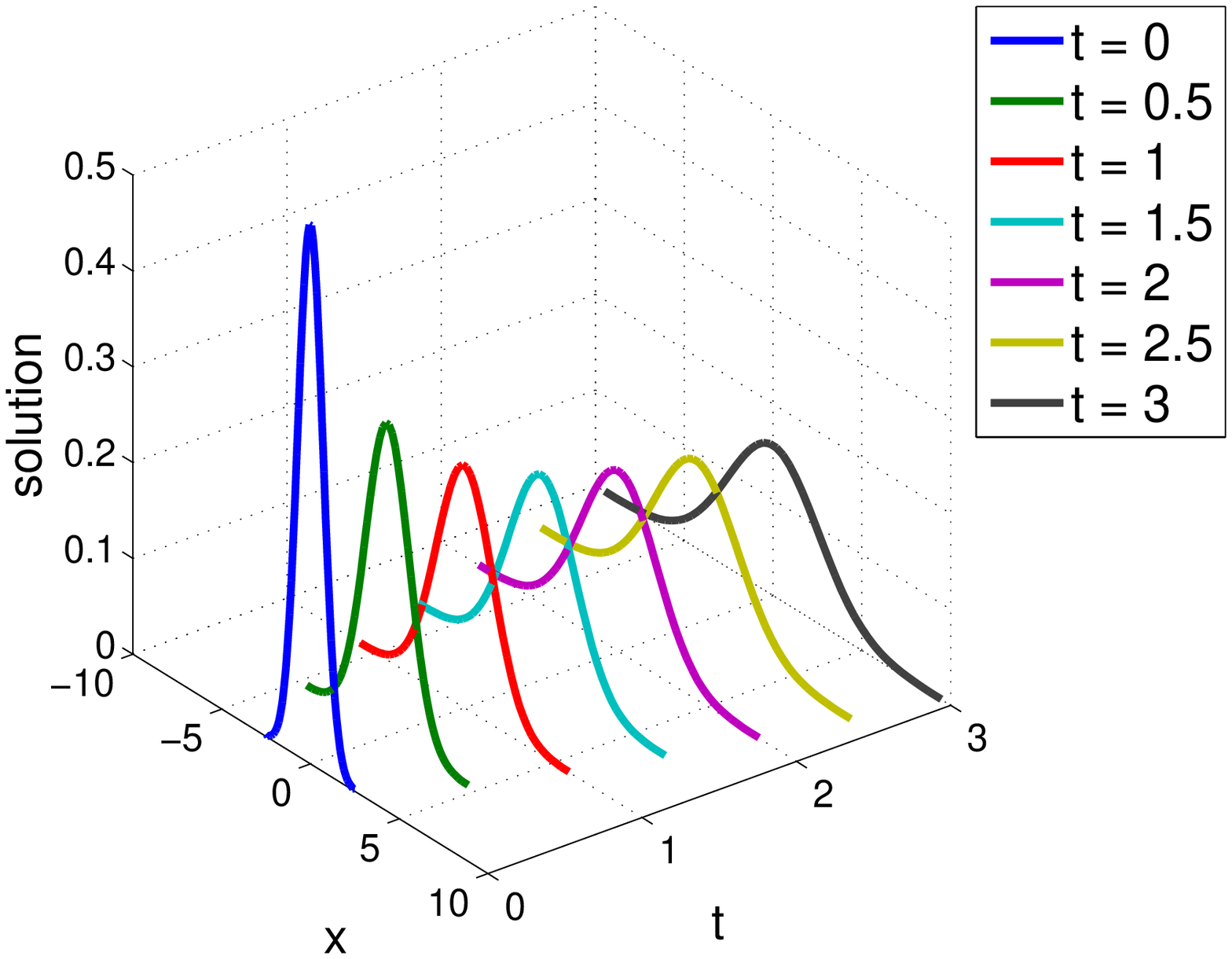}}
	\subcaptionbox{Error with $N$ at $T=3$ with $\Delta t = 5\cdot 10^{-7}$ .\label{fig:heat-error}}{\includegraphics[scale=0.41]{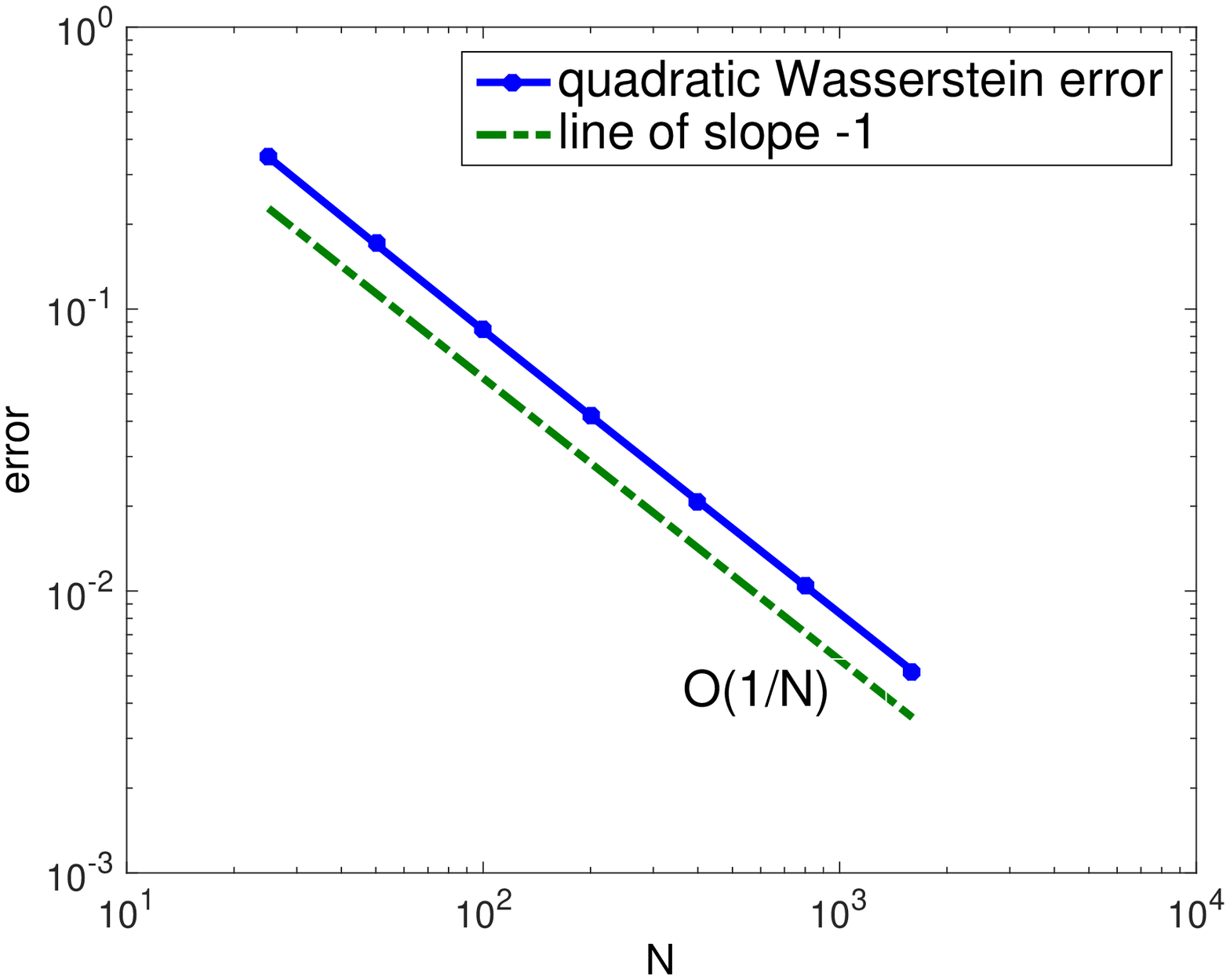}}
	\caption{The heat equation.\label{fig:heat}}
\efig

From Figure \ref{fig:heat-error} we can see that the quadratic Wasserstein error with respect to the exact solution $\rho$ evolves linearly with the number of particles on a log-log plot. From this plot, it looks fair to say that the error of our scheme \eqref{eq:jko-final-explicit} is $e_{d_2} = \mathcal{O}\left(\frac1N\right)$.

\subsubsection{The porous medium equation}
\label{subsec:porous-medium-equation-1d}

The porous medium equation $\rho_t=\Delta \rho^m$ is a continuum gradient flow \eqref{eq:gradient-flow} for $H(\rho) = \frac{\rho^m}{m-1}, m > 1$, and $V = W = 0$. The evolution of the solution starting from the initial data
\be\label{eq:rho0porous}
	\rho_0^\mt{por}(x) = \frac{1}{t_0^\alpha} \psi\left(\frac{x}{t_0^\alpha}\right) \quad \mbox{with $t_0 = 0.25$}
\ee
is shown in Figure \ref{fig:porous-evo}. Here $\alpha = \frac{1}{m+1}$, $\psi \colon \xi \mapsto (K - \kappa\xi^2)_+^{1/(m-1)}$, where the subscript $_+$ stands for the positive part, $\kappa = \frac{m-1}{2m(m+1)}$ and $K$ determined by the conservation of mass. Indeed, since the total conserved mass is one, then the constant $K$ can be expressed as
\bes
	\textstyle{K = \left[\Gamma\left(\frac{1}{m-1} + \frac{3}{2}\right)\sqrt{\kappa}\bigg/\left(\Gamma\left(\frac{m}{m-1}\right)\Gamma\left(\frac{1}{2}\right)\right)\right]^{\frac{2(m-1)}{m+1}}},
\ees
where $\Gamma$ is the Gamma-function. Then one can verify that, for all $t>0$,
\bes
\rho(t,x) = \dfrac{1}{(t+t_0)^\alpha} \psi\left(\dfrac{x}{(t+t_0)^\alpha}\right) \quad \mbox{for all $x \in \R$},
\ees
is a solution, see \cite[Section 4.4]{Vazquez}. For $x \in \left(0,(t+t_0)^\alpha\sqrt{\frac K\kappa}\right]$, the cumulative distribution is
\bes
	F(t,x) = \frac{1}{2} + \int_0^x \frac{1}{(t+t_0)^\alpha}\left(K - \frac{\kappa x^2 }{(t+t_0)^{2\alpha}}\right)_+^{\frac{1}{m-1}} \d x = \frac{1}{2} + \frac{1}{2}I\left( \frac{\kappa}{K}\frac{x^2}{(t+t_0)^{2\alpha}}; \frac{1}{2},\frac{m}{m-1}\right),
\ees
where $I(x;a,b) := \frac{B(x;a,b)}{B(1;a,b)}$ and $B$ is the incomplete Beta-function, that is,
\bes
	B(x;a,b) = \int_0^x z^{a-1}(1-z)^{b-1} \d z \quad \mbox{for all $x \geq 0$ and $a,b > 0$}.
\ees
Similarly, for $x \in \left(-(t+t_0)^\alpha\sqrt{\frac K\kappa},0\right]$,
\bes
	F(t,x) = \frac{1}{2} - \frac{1}{2}I\left( \frac{\kappa}{K}\frac{x^2}{(t+t_0)^{2\alpha}}; \frac{1}{2},\frac{m}{m-1}\right).
\ees
Also, noticing that
\bes
	\textstyle{K^{\frac{1}{m-1} + \frac{1}{2}}\kappa^{-\frac{1}{2}} = \left(B\left(1;\frac{1}{2},\frac{m}{m-1}\right)\right)^{-1}},
\ees
we get that the pseudo-inverse of $F$ is
\bes
	\Phi(t,\e) =
	\begin{cases}
		-(t+t_0)^\alpha \sqrt{\dfrac{K}{\kappa}}\sqrt{I^{-1}\left(1 - 2\e; \dfrac{1}{2},\dfrac{m}{m-1}\right)} & \mbox{if $0 \leq \e < \frac 12$}, \\[0.4cm]
 		(t+t_0)^\alpha\sqrt{\dfrac{K}{\kappa}}\sqrt{I^{-1}\left(2\e - 1; \dfrac{1}{2},\dfrac{m}{m-1}\right)} & \mbox{if $\frac 12 \leq \e < 1$}.
\end{cases}
\ees
The error can then be found using \eqref{eq:error-wasserstein-1d-2} and approximating the integrals therein, see Figure \ref{fig:porous-error}. In Figure \ref{fig:porous} we chose $I_\mt{init} = \left[-k_0,k_0\right]$, where $k_0:=t_0^\alpha\sqrt{K/\kappa}$, see Section \ref{subsubsec:initially-equally-spaced-particles}.

\bfig[!ht]
\centering
	\subcaptionbox{Evolution for $N=50$ with $\Delta t = \frac{0.1}{N^2}$.\label{fig:porous-evo}}{\includegraphics[scale=0.41]{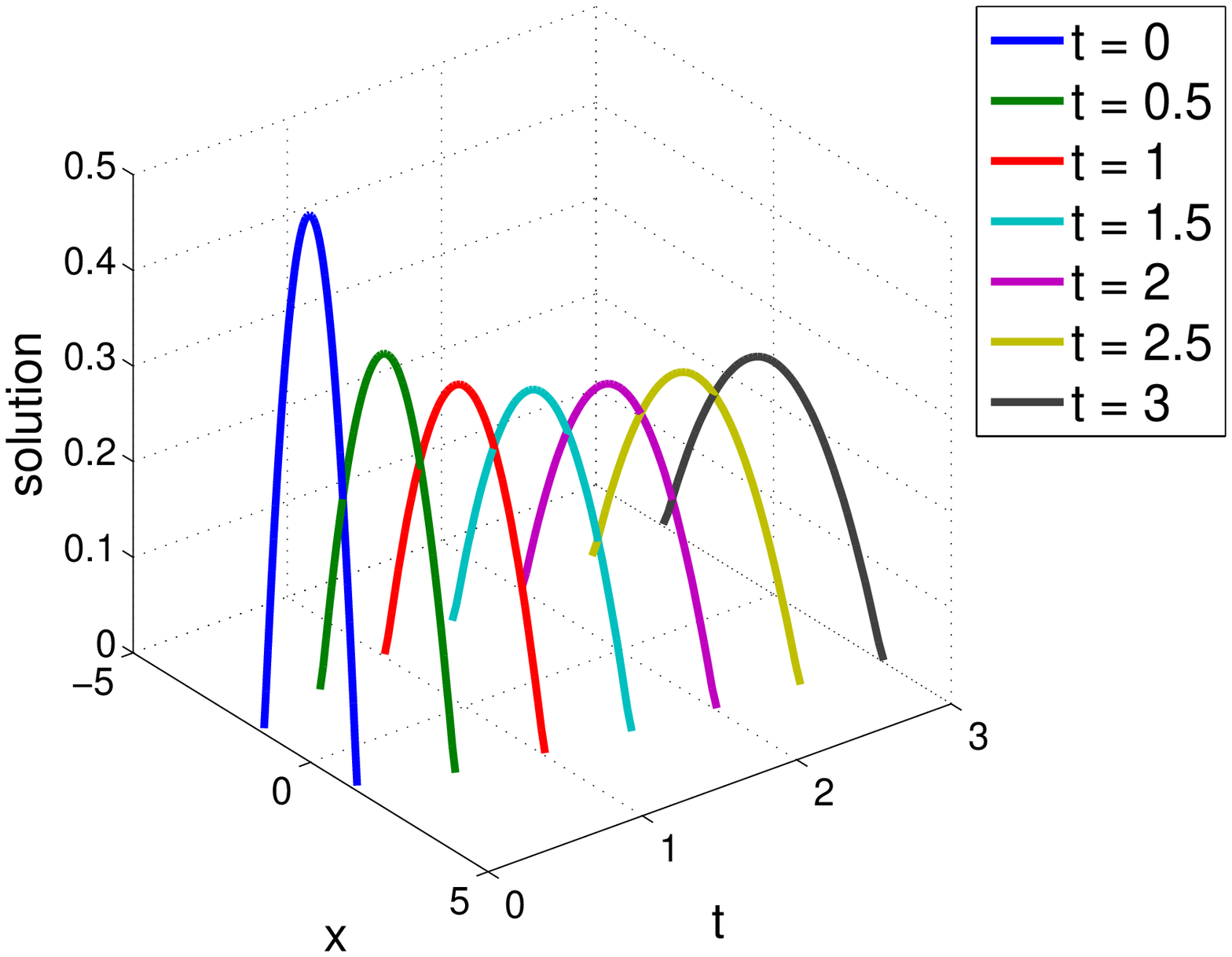}}
	\subcaptionbox{Error with $N$ at $T=3$ with $\Delta t = 10^{-7}$ .\label{fig:porous-error}}{\includegraphics[scale=0.41]{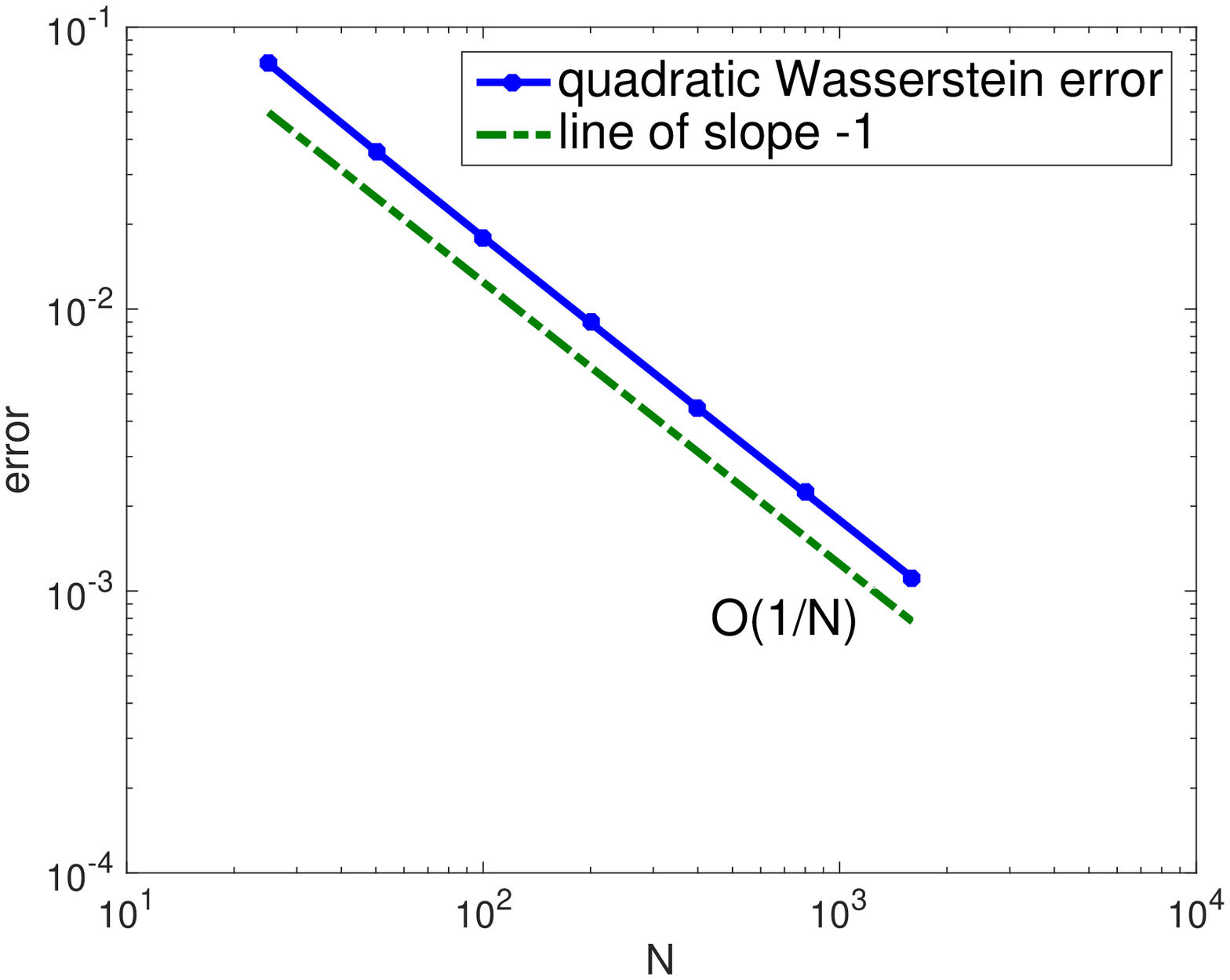}}
	\caption{The porous medium equation with $m=2$.\label{fig:porous}}
\efig
From the plot in Figure \ref{fig:porous-error}, as for the one in Figure \ref{fig:heat-error} for the heat equation, it looks fair to say again that the error of our scheme \eqref{eq:jko-final-explicit} is $e_{d_2} = \mathcal{O}\left(\frac1N\right)$.

\subsubsection{The linear Fokker-Planck equation}
\label{subsec:heat-equation-confinement-1d}

Let us consider the heat equation with quadratic confinement potential, i.e., the heat equation with $V(x) = \frac{x^2}{2}$. In this case, regardless of the initial condition, there is a steady state:
\bes
	\rho_\infty(x) := \frac{1}{\sqrt{2\pi}} e^{-\frac{x^2}{2}} \quad \mbox{for all $x \in \R$}.
\ees
Figure \ref{fig:heat-confined-error-discrete} was obtained from the continuum initial profile $\rho_0^\mt{heat}$, see \eqref{eq:rho0heat}, with $I_\mt{init} = [-2.5,2.5]$ for Figure \ref{fig:heat-confined-evo} and $I_\mt{init} = [-4,4]$ for Figure \ref{fig:heat-confined-error-discrete}.

\bfig[!ht]
\centering
	\subcaptionbox{Comparison at $T=4$ for $N=50$.\label{fig:heat-confined-evo}}{\includegraphics[scale=0.41]{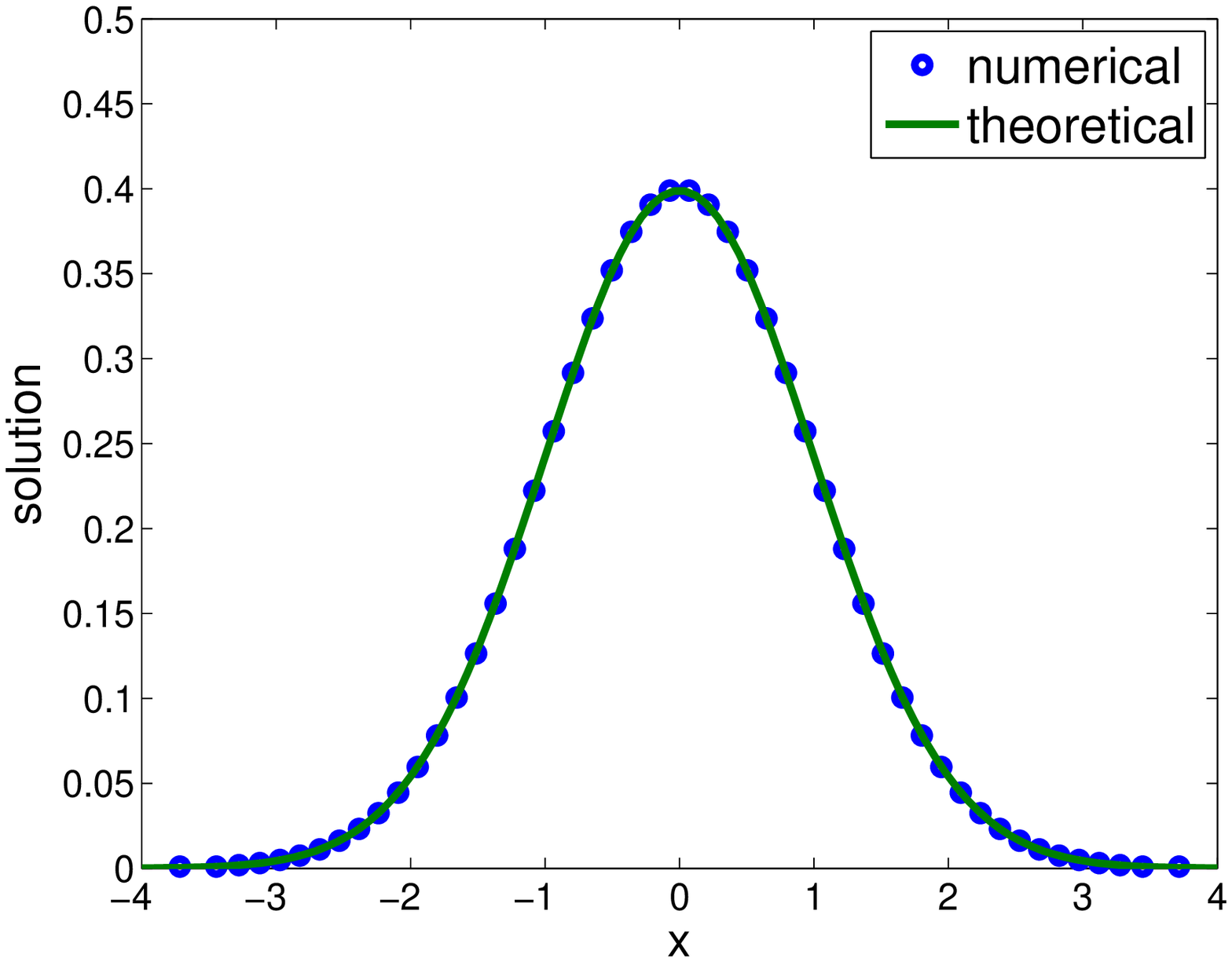}}
	\subcaptionbox{Error with final time $T$ for different $N$'s with respect to the discrete steady state.\label{fig:heat-confined-error-discrete}}{\includegraphics[scale=0.41]{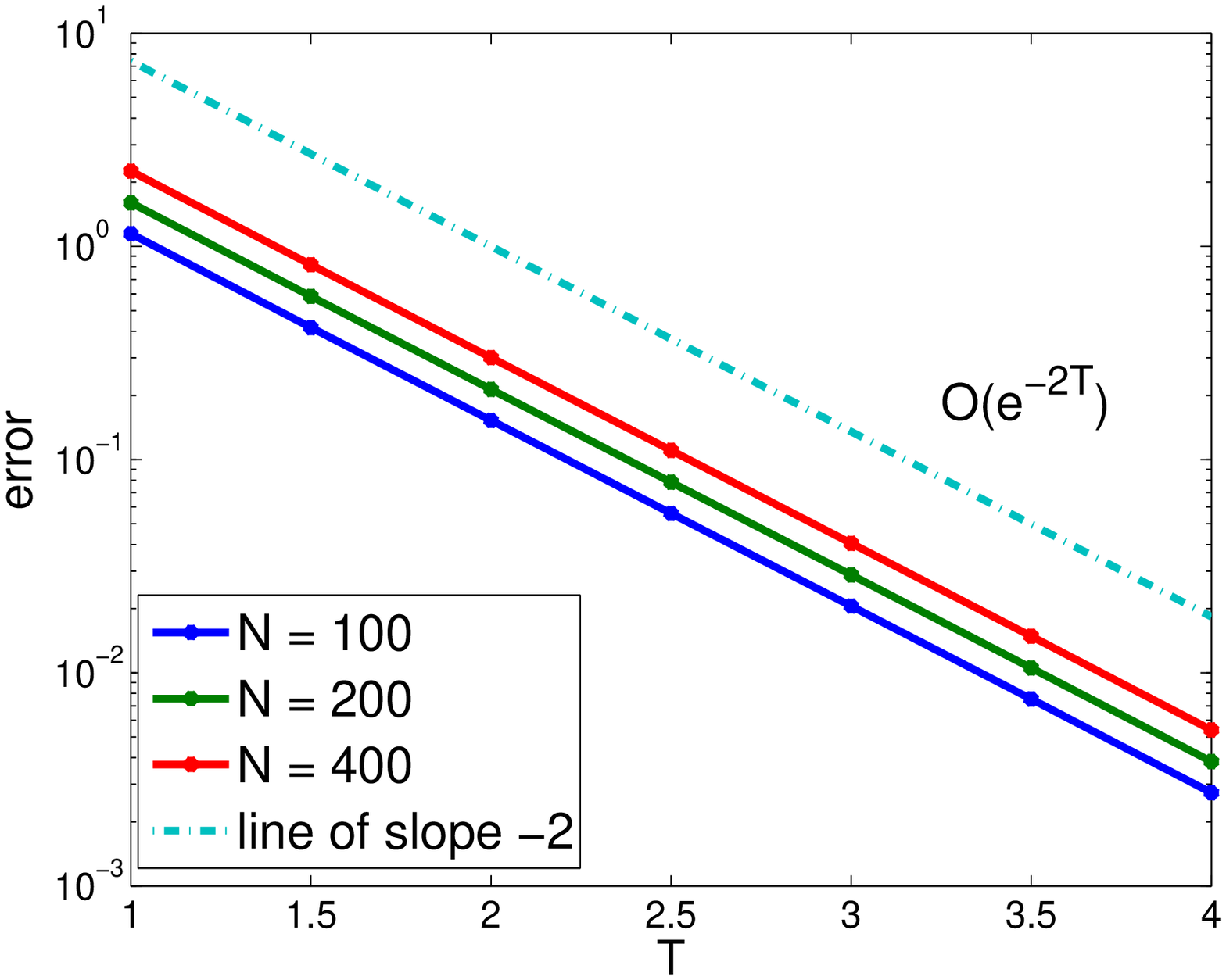}}
	\caption{The linear Fokker-Planck equation with $\Delta t = 10^{-5}$ -- Stabilisation in time of the scheme (rate of convergence to the discrete steady state).\label{fig:heat-confined}}
\efig

In Figure \ref{fig:heat-confined-error-discrete}, we can see that the stabilisation of the scheme towards the discrete steady state (which we arbitrarily define as being the discrete solution obtained at $T = 6$) is linear on a semi-log plot with a slope very close to $-2$, which seems not to depend on the number of particles $N$. We can therefore write the error as $e_{d_2}^* = \mathcal{O}\left(e^{-2T}\right)$, see \eqref{eq:error-discrete}.

\subsubsection{The nonlinear Fokker-Planck equation}
\label{subsec:porous-medium-equation-confinement-1d}

Let us consider the porous medium equation with quadratic confinement potential, i.e., the porous medium equation with $V(x) = \frac{x^2}{2}$. In this case, regardless of the initial profile and since the total mass is one, the steady state is
\bes
	\rho_\infty(x) := A (R^2 - x^2)_+^{\frac{1}{m-1}} \quad \mbox{for all $x\in\R$},
\ees
where $A = (\frac{m-1}{2m})^{1/(m-1)}$ and $R = (AB(1;\frac 12,\frac{m}{m-1}))^{(1-m)/(m+1)}$. Figure \ref{fig:porous-confined-error-discrete} was obtained from the continuum initial profile $\rho_0^\mt{por}$, see \eqref{eq:rho0porous}, with $I_\mt{init} = \left[-k_0,k_0\right]$, where again $k_0:=t_0^\alpha\sqrt{K/\kappa}$.

\bfig[!ht]
\centering
	\subcaptionbox{Comparison at $T=4$ for $N=50$.\label{fig:porous-confined-evo}}{\includegraphics[scale=0.41]{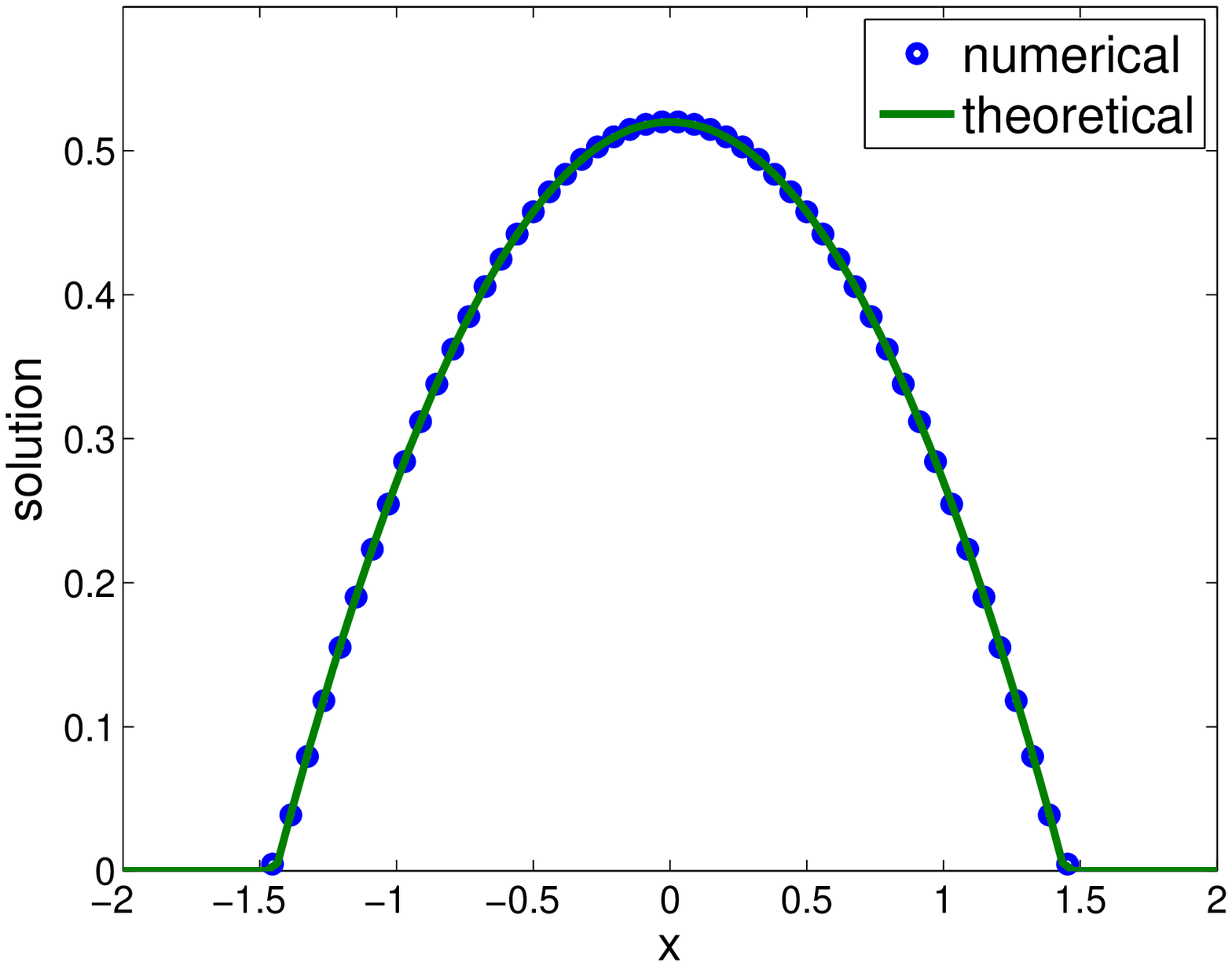}}
	\subcaptionbox{Error with the final time $T$ for different $N$'s with respect to the discrete steady state.\label{fig:porous-confined-error-discrete}}{\includegraphics[scale=0.41]{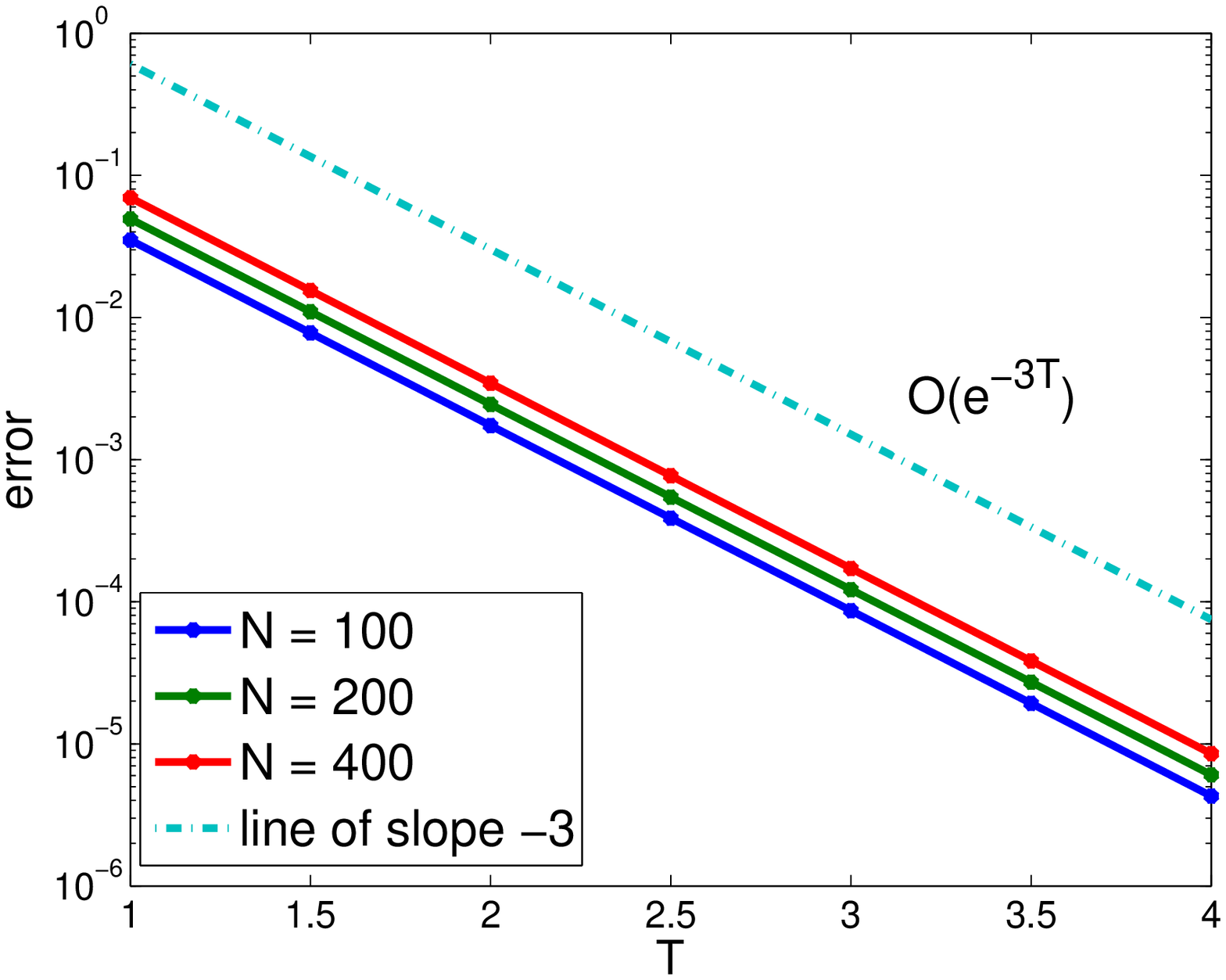}}
	\caption{The nonlinear Fokker-Planck equation with $m=2$ and $\Delta t = \frac{0.1}{N^2}$ -- Stabilisation in time of the scheme (rate of convergence to the discrete steady state).\label{fig:porous-confined}}
\efig

As already noted from Figure \ref{fig:heat-confined-error-discrete}, for the heat equation the numerical error can be written $e_{d_2}^* = \mathcal{O}\left( e^{-2T}\right)$, as it is actually expected from the theory; for the porous medium equation the theory says that we should expect $e_{d_2}^* =  \mathcal{O}\left( e^{-(m+1)T}\right)$, see \cite{CDT} and references therein, which is nicely confirmed numerically by Figure \ref{fig:porous-confined-error-discrete} with $m=2$.

\subsubsection{A two-dimensional test: the heat equation}
\label{subsubsec:two-dimensions}
It is straightforward to generalise the scheme~\eqref{eq:jko-final-explicit}, derived in Section \ref{subsec:jko}, to higher dimensions whenever the expression~\eqref{eq:distance} for the Wasserstein distance can be used, which is the case when the approximation $\mu^{n+1}$ at time step $n+1$ is sufficiently close to $\mu^n$, i.e., when the time step $\Delta t$ is small enough. Indeed, let us emphasise that the Wasserstein distance is computed between empirical measures on points, and not between their approximations on non-overlapping balls (which are only used to define the diffusion part of the regularised discrete energy functional, see \eqref{eq:energy-discrete}). When the time step is small enough, the Wasserstein distance approximation \eqref{eq:distance} is exact, since then the splitting of mass between empirical measures possibly happening in higher-dimensional optimal transport actually does not occur.

We test our scheme for the heat equation in two dimensions. The initial continuum density is $\rho_0(x)=\frac{1}{4\pi t_0}e^{-|x|^2/(4t_0)}$ with $t_0 = 0.125$, and the particle positions at $T=1$ are shown in Figure~\ref{fig:heat2dfinal}. The data were initialised by fixing the particles on a regular grid as in Figure~\ref{fig:heat2dinitial}, with weights being the integrals of the continuum density $\rho_0$ on the Voronoi cells generated from the particles. 

\bfig[!ht]
\centering
	\subcaptionbox{Initial time.\label{fig:heat2dinitial}}{\includegraphics[scale=0.42]{./twod_init}}
	\subcaptionbox{Final time $T=1$.\label{fig:heat2dfinal}}{\includegraphics[scale=0.42]{./twod_final}}
	\caption{The particles' positions for the two-dimensional heat equation for $N=100$ with $\Delta t = 10^{-4}$.\label{fig:heat2dpart}}
\end{figure}

\bfig[!ht]
\centering
	\subcaptionbox{Evolution of the second moment, $\int |x|^2 \rho$.}{\includegraphics[scale=0.42]{./twomoment}}
	\subcaptionbox{Evolution of the entropy, $\int \rho\log \rho$.}{\includegraphics[scale=0.42]{./twoentropy}}
	\caption{Accuracy for the two-dimensional heat equation with $\Delta t = 10^{-4}$. \label{fig:heat2dstat}}
\end{figure}

The averaged quantities along the time evolution, like the second moments $\sum_{i=1}^N w_i|x_i|^2$ and the entropy $\sum_{i=1}^N w_i \log \frac{w_i}{|B_i|}$, seem to be very accurate, as shown in Figure~\ref{fig:heat2dstat}, and this accuracy does not seem to degenerate with time. However, representing the numerical solution in this two-dimensional test is a delicate issue since gaps between discretisation balls are significant; this is an issue in itself which we do not deal with here.

\section{Aggregation-Diffusion Equations}
\label{sec:aggregation-diffusion}

\subsection{The modified Keller-Segel equation}
\label{subsec:keller-segel}

We start by considering a modified one-dimensional Keller-Segel equation, that is the continuum gradient flow \eqref{eq:gradient-flow} with $H(\rho) = \rho\log \rho$ and $W(x) = 2\chi\log|x|$ (and $W(0) := 0$), where $\chi > 0$ is a parameter quantifying the attraction (see \cite{CPS,Blanchet} for well-posedness and qualitative properties, and \cite{CG} for the approximation of this equation by a different particle method).

This model shows a critical behaviour depending only on the chemosensitivity strength $\chi$ as the classical Keller-Segel model in two dimensions, that is, there is a dichotomy of blow-up in finite time or global existence which is only determined by $\chi$ being larger or less than 1, see \cite{CPS,Blanchet}. In case $\chi<1$ solutions spread in time behaving like self-similar solutions. To get the leading order profile given by the self-similar solution, a time-space scaling is done for $\chi <1$, which is equivalent to impose a quadratic confinement potential on particles, see \cite{Blanchet}. The long-time behaviour in this subcritical case is given by the profile of the self-similar solution.

\subsubsection{Theoretical properties}
\label{subsubsec:theoretical-properties}

We show that our particle approximation keeps approximately the criticality of the original Keller-Segel model at the discrete level. We show that there exist two positive constants $\chi_1$ and $\chi_2(N)$ such that the following holds: if an appropriate confinement potential $V$ is considered, the discrete and the $p$-approximated discrete gradient flows \eqref{eq:discrete-gradient-flow-inclusion} and \eqref{eq:discrete-gradient-flow-ode} for the modified Keller-Segel equation have steady states for $\chi < \chi_1$; while if $V=0$, the $p$-approximated discrete gradient flow shows finite-time blow-up for $\chi>\chi_2(N)$. Quite surprisingly, $\chi_1$ happens to be exactly the critical parameter at the continuum level, i.e., $\chi_1 = 1$, and $\chi_2(N)$ tends to 1 as $N\to\infty$ and does not depend on $p$ (but only on $N$ and the chosen weights). 

By the term ``blow-up" at the discrete level we mean the event of two, or more, particles colliding. Also note that in the following the term ``maximal time of existence" indicates either the first time when two or more particles of a solution collide, i.e., the first blow-up time, or the first time when the norm of a solution equals $+\infty$.

First, let us prove that the discrete and $p$-approximated discrete confined Keller-Segel equations show no collisions of particles if $\chi < \chi_1$.

\begin{prop} \label{prop:no-collision}
Consider the discrete gradient flow corresponding to the confined Keller-Segel equation with $V$ coercive and such that the function $x \mapsto \inf_{y\in\R} (w_1V(y) + w_NV(x+y)) - \log |x|$ is coercive. Suppose there exists a solution $\bx$ to such gradient flow, emanating from an initial condition $\boldsymbol{x^0} \in \R_w^N$, up to some maximal time of existence, say $T^* > 0$. If
$$
	\chi<\chi_1 :=1,
$$
then no particles of $\bx$ can collide in $[0,T^*)$; furthermore, the minimal inter-particle distance is uniformly bounded from below in time by a positive constant.
\end{prop}
\begin{proof}
The energy $\wtE$ is a Lyapunov functional, i.e.,
\be\label{eq:decreasing-functional}
	\wtE(\bx(t)) \leq \wtE(\boldsymbol{x^0}) := E_0< +\infty \quad \mbox{for all $t\in [0,T^*)$},
\ee
see \cite[Theorem 1 of Section 3.4]{Aubin}. Fix $t\in [0,T^*)$ and get, by \eqref{eq:energy-discrete},
\begin{align*} \label{eq:discrete-functional-ks}
	\wtE(\bx(t)) &= \sum_{i=1}^N w_i\log w_i - \sum_{i=1}^N w_i \log r_i(t) + \sum_{i=1}^N w_i V(x_i(t))\\
	&\phantom{={}} + \chi \sum_{i=1}^N\sum_{\substack{j=1\\j\neq i}}^N w_i w_j \log|x_i(t) - x_j(t)| \numberthis ,
\end{align*}
where $r_i(t) = \min(\Delta x_i(t),\Delta x_{i+1}(t))$. Writing $\log_-x := \log x$ if $0< x < 1$ and $\log_-x:= 0$ if $x\geq 1$, and using that $\log$ is increasing,
\begin{align*}
	\sum_{i=1}^N\sum_{\substack{j=1\\j\neq i}}^N w_i w_j \log|x_i(t) - x_j(t)| &\geq \sum_{i=1}^N\sum_{\substack{j=1\\j\neq i}}^N w_i w_j \log\underset{\substack{k\in\{1,\dots,N\}\\ k\neq i}}{\min}|x_i(t) - x_k(t)|\\
	&\geq \sum_{i=1}^N w_i (1-w_i) \log_- r_i(t) \geq \sum_{i=1}^N w_i\log_- r_i(t).
\end{align*}
Hence, writing $\log_+x := \log x$ if $x \geq 1$ and $\log_+x:= 0$ if $0\leq x < 1$, and by \eqref{eq:discrete-functional-ks},
\begin{align*}
	\wtE(\bx(t)) &\geq \sum_{i=1}^N w_i\log w_i + \sum_{i=1}^N w_i V(x_i(t)) + (\chi-1)\sum_{i=1}^N w_i \log_- r_i(t) - \sum_{i=1}^N w_i \log_+r_i(t)\\
	&\geq \sum_{i=1}^N w_i\log w_i + \sum_{i=2}^{N-1} w_i V(x_i(t)) + w_1V(x_1(t)) + w_NV(x_N(t))\\
	&\phantom{={}}+ (\chi-1)\sum_{i=1}^N w_i \log_- r_i(t) - \log_+(x_N(t) - x_1(t)),
\end{align*}
using that $-\log_+r_i(t) \geq -\log_+(x_N(t) - x_1(t))$ for all $i\in \{1,\dots,N\}$. From the assumptions on $V$ we know that $V$ is bounded from below and also that $w_1V(x_1(t)) + w_NV(x_N(t)) - \log_+(x_N(t) - x_1(t))$ is bounded from below uniformly in time. Therefore there exists a constant $C\in\R$, independent of $t$, such that
\be\label{boundbelowEKS}
	E_0 \geq \wtE(\bx(t)) \geq (\chi-1)\sum_{i=1}^N w_i \log_- r_i(t) + C,
\ee
using \eqref{eq:decreasing-functional}. We see that if $\chi < \chi_1 = 1$, the minimal inter-particle distances $r_i(t)$ cannot get arbitrarily small, or the energy $\wtE(\bx(t))$ gets larger than its initial value $E_0$, which violates the fact that the system is a gradient flow. Hence the result.
\end{proof}
\begin{rem}
	The proof above is given only for the discrete case; however, note that it can be easily adapted to the $p$-approximated one, if $p\geq1$, by Proposition \ref{prop:min-p}\ref{it:above} with $s=2$, and without the need to change the constant $\chi_1$.
\end{rem}

We can now show the global existence in time and the existence of steady states for the discrete and $p$-approximated discrete confined Keller-Segel equations.

\begin{prop} \label{prop:existence-steady-states}
Consider the discrete gradient flow corresponding to the confined Keller-Segel equation with $V$ satisfying the same hypotheses as in Proposition \ref{prop:no-collision}. If $\chi<\chi_1$, then any solution to this gradient flow, if it exists, exists globally in time and the gradient flow has a steady state.
\end{prop}
\begin{proof}
Let $\chi < \chi_1 = 1$. Suppose there exists such a solution $\bx$, emanating from an initial condition $\boldsymbol{x^0} \in \R_w^N$, defined on $[0,T^*)$, see Proposition \ref{prop:no-collision}, and assume that $x_N(t) - x_1(t) \to +\infty$ as $t\to T^*$. For all $t\in[0,T^*)$, the proof of Proposition \ref{prop:no-collision} implies that there exists a $t$-independent constant $C_1\in\R$ such that
\bes
	E_0 \geq \wtE(\bx(t)) \geq (\chi-1)\sum_{i=1}^N w_i \log_- r_i(t) + f(x_N(t)-x_1(t)) + C_1 \geq f(x_N(t)-x_1(t)) + C_1,
\ees
since $(\chi-1)\sum_{i=1}^N w_i \log_- r_i(t)\geq0$, and where $f(x_N(t) - x_1(t)) := \inf_{y\in\R} (w_1V(y) + w_NV(x_N(t)-x_1(t) +y)) - \log_+(x_N(t) - x_1(t))$. By the growth assumption at infinity on $V$, we have $f(x_N(t) - x_1(t)) \to +\infty$ as $t\to T^*$, which implies that the inequality above is violated for a time $t$ close enough to $T^*$. Therefore $x_N(t) - x_1(t)$ cannot diverge as $t\to T^*$, and thus there exists $C_2\in\R$, uniform in $t$, such that $\log_+ (x_N(t) - x_1(t)) \leq C_2$ for all $t \in [0,T^*)$. Therefore,
\bes
	E_0 \geq \wtE(\bx(t)) \geq w_1V(x_1(t)) + w_N V(x_N(t)) - C_2 + C_1,
\ees
which, by coercivity of $V$ shows that $x_1(t)$ and $x_N(t)$ cannot diverge. Thus, there exists some constant $\ell>0$, independent of time, such that $\{x_1(t),\dots,x_N(t)\} \subset B(0,\ell)$ for all $t\in[0,T^*)$. This, together with the no-collision result in Proposition \ref{prop:no-collision}, shows that the maximal time of existence of the discrete gradient flow solution is $T^*=\infty$.

Finally, the functional $\wtE$ is lower semi-continous and bounded from below on any sublevel set of $\wtE$ which are compact due to \eqref{boundbelowEKS}. Indeed, the previous argument shows that $x_i\in B(0,\ell)$ for all $\bx \in \R_w^N$ such that $E_0\geq \wtE(\bx)$. Moreover, the same argument leading to \eqref{boundbelowEKS} implies that
$$
E_0\geq \wtE(\bx) \geq  C,
$$
for all $\bx \in \R_w^N$ such that $E_0\geq \wtE(\bx)$, since $\chi<1$. Therefore, by a direct method of calculus of variations, we get that $\wtE$ has a global minimiser, which ends the proof.
\end{proof}

\begin{rem}
	Similarly to Proposition \ref{prop:existence-steady-states} the proof above is given only for the discrete case, but is easily adaptable to the $p$-approximated one if $p\geq1$, by Proposition \ref{prop:min-p}\ref{it:above} with $s=2$.
\end{rem}

\begin{rem}
	The assumptions on $V$ of Propositions \ref{prop:no-collision} and \ref{prop:existence-steady-states} are in particular satisfied by $V(x) = \frac{x^k}{k}$ for any $k\geq1$. Propositions \ref{prop:no-collision} and \ref{prop:existence-steady-states} are also true for any $W$ bounded from below and satisfying Hypothesis \ref{hyp:VW}, with no required constraint on $\chi$; in particular this is the case for the linear Fokker-Planck equation, that is with $V(x) = \frac{x^2}{2}$ and $W = 0$.
\end{rem}

Let us now turn to the supercritical case. In the unconfined continuum modified Keller-Segel equation, it is known that solutions blow up in finite time if $\chi > 1$. The proof of non-existence of global-in-time solutions is obtained by computing the evolution of the second moment $M_2(t)$ of solutions $\rho(t)$ at any time $t > 0$. Then, a formal computation leads to
\be\label{eq:cont-moment}
	\dfrac{\mathrm{d}M_2}{\mathrm{d}t}(t) = \dfrac{\mathrm{d}}{\mathrm{d}t} \int_\R x^2 \d \rho(t,x) = 2(1-\chi).
\ee
Therefore the evolution of the second moment is linear in time with slope $2(1-\chi)$. This slope is negative if $\chi > 1$, which implies that $M_2(t)$ becomes zero in finite time leading to concentration of mass in finite time and contradiction with the assumption of global existence. We want here to show that our $p$-approximated discrete gradient flow \eqref{eq:discrete-gradient-flow-ode} preserves this finite-time blow-up property for some numerical critical parameter $\chi_2(N)$, at least when all particles have same weight. Recall that, at the discrete level, we define blow-up as being the event of two particles colliding.

\begin{prop} \label{prop:keller-segel-blow-up}
Let $p > 0$ and consider the $p$-approximated discrete unconfined Keller-Segel gradient flow with $w_i = 1/N$ for all $i\in \{1,\dots,N\}$, on the whole time line $[0,\infty)$. All solutions blow up in finite time if $\chi$ is greater than
\bes
	\chi_2(N) := 1 + \frac{1}{N-1}.
\ees
\end{prop}
\begin{proof}
	Suppose that there exists $\bx$, a $p$-approximated discrete Keller-Segel gradient flow solution emanating from an initial condition $\boldsymbol{x^0} \in \R_w^N$, defined on some maximal interval of existence $[0,T^*)$. Let us compute the evolution of the second moment of $\mu_N$, the empirical measure associated to $\bx$, at any $t\in[0,T^*)$.
\be \label{eq:moment}
	\dfrac{\der M_2}{\der t}(t) = \dfrac{\der}{\der t} \int_\R x^2 \d\mu_N(t,x) = \dfrac{\der}{\der t} \dfrac 1N \sum_{i=1}^N x_i^2(t) = \dfrac{2}{N} \sum_{i=1}^N x_i(t) \dfrac{\der x_i}{\der t}(t).
\ee
In the following we drop the dependencies on time for the sake of simplicity.

Suppose $T^* = \infty$. We want to find a contradiction if $\chi > \chi_2(N)$ by computing explicitly the evolution of the second moment in \eqref{eq:moment}. Write $\Delta_i^j:=1+\left(\Delta x_j/\Delta x_i\right)^p$ for all $i,j\in\{1,\dots,N\}$ with $i\neq j$, recalling the convention $\Delta x_1 = \Delta x_{N+1} = +\infty$ and setting $\Delta_0^1 = \Delta_{N+2}^{N+1} = +\infty$. By \eqref{eq:discrete-gradient-flow-ode},
\begin{align*}
	\sum_{i=1}^N x_i \dfrac{\der x_i}{\der t} &= \underbrace{\sum_{i=1}^N \left(\left(\dfrac{x_i/\Delta x_i}{\Delta_{i-1}^i} + \dfrac{x_i/\Delta x_i}{\Delta_{i+1}^i}\right) - \left(\dfrac{x_i/\Delta x_{i+1}}{\Delta_i^{i+1}} + \dfrac{x_i/\Delta x_{i+1}}{\Delta_{i+2}^{i+1}}\right)\right)}_{=:A_1}\\
	&\phantom{={}}- \dfrac{2\chi}{N} \underbrace{\sum_{i=1}^N x_i \sum_{\substack{j=1\\j\neq i}}^N \bar W'(x_i - x_j)}_{=:A_2},
\end{align*}
where $\bar W := \log|\cdot|$. First, compute $A_1$ by appropriately rearranging the sum terms,
\begin{align*}
	A_1 &=  \sum_{i = 3}^{N-2} \left(\left(\dfrac{x_{i}/\Delta x_{i}}{\Delta_{i-1}^i} + \dfrac{x_{i}/\Delta x_{i}}{\Delta_{i+1}^i}\right)-\left(\dfrac{x_{i}/\Delta x_{i+1}}{\Delta_i^{i+1}} + \dfrac{x_{i}/\Delta x_{i+1}}{\Delta_{i+2}^{i+1}}\right)\right) + \left(1 + \dfrac{1}{\Delta_3^2} \right)\\
	&\phantom{={}}+ \left(1 + \dfrac{1}{\Delta_{N-1}^N} \right) - \left( \dfrac{x_{2}/\Delta x_{3}}{\Delta_2^3}+ \dfrac{x_{2}/\Delta x_{3}}{\Delta_4^3} \right) + \left( \dfrac{x_{N-1}/\Delta x_{N-1}}{\Delta_N^{N-1}}+ \dfrac{x_{N-1}/\Delta x_{N-1}}{\Delta_{N-2}^{N-1}} \right).
\end{align*}
Also,
\begin{align*}
	\sum_{i = 3}^{N-2} & \left(\left(\dfrac{x_{i}/\Delta x_{i} }{\Delta_{i-1}^i} + \dfrac{x_{i}/\Delta x_{i} }{\Delta_{i+1}^i}\right) - \left(\dfrac{x_{i}/\Delta x_{i+1}}{\Delta_i^{i+1}} + \dfrac{x_{i}/\Delta x_{i+1}}{\Delta_{i+2}^{i+1}}\right) \right)\\
	&= \left( \dfrac{x_{3}/\Delta x_{3}}{\Delta_4^3}+ \dfrac{x_{3}/\Delta x_{3}}{\Delta_2^3} \right) - \left( \dfrac{x_{N-1}/\Delta x_{N-1}}{\Delta_{N-2}^{N-1}}+ \dfrac{x_{N-1}/\Delta x_{N-1}}{\Delta_N^{N-1}} \right) + \sum_{i = 3}^{N-2} \left(\dfrac{1}{\Delta_i^{i+1}} + \dfrac{1}{\Delta_{i+2}^{i+1}}\right).
\end{align*}
Hence, by combining the last two computations, we get
\begin{align*}
	A_1 &= \sum_{i = 3}^{N-2} \left(\dfrac{1}{\Delta_i^{i+1}} + \dfrac{1}{\Delta_{i+2}^{i+1}}\right) + \left( \dfrac{1}{\Delta_4^3}+ \dfrac{1}{\Delta_2^3} \right)+ \left(1 + \dfrac{1}{\Delta_3^2}\right) +\left(1 + \dfrac{1}{\Delta_{N-1}^N}\right)\\
	&= \sum_{i = 2}^{N-1} \left(\dfrac{1}{\Delta_i^{i+1}} + \dfrac{1}{\Delta_{i+1}^i}\right) + 2 = \sum_{i = 2}^{N-1} 1 + 2 = N.
\end{align*}
Then, compute $A_2$ by using the anti-symmetry of $\bar W'(x) = \frac1x$,
\begin{align*}
	A_2 &= (x_1 - x_2) \bar W'(x_1 - x_2) + (x_1-x_3)\bar W'(x_1-x_3) + \dots + (x_1-x_N) \bar W'(x_1-x_N)\\
	& + (x_2 - x_3) \bar W'(x_2 - x_3) + (x_2-x_4)\bar W'(x_2-x_4) + \dots + (x_2-x_N)\bar W'(x_2-x_N) + \dots\\
	& + (x_{N-1} - x_N) \bar W'(x_{N-1}-x_N) = (N-1) + (N-2) + \dots + 1 = \textstyle{\sum_{i=1}^{N-1} (N-i) = \frac{(N-1)N}{2}}.
\end{align*}
Therefore, for all $t \in[0,T^*) = [0,\infty)$,
\be\label{eq:discrete-moment}
	\dfrac{\der M_2}{\der t}(t) = \dfrac{2}{N} \left( N - \dfrac{2\chi}{N} \dfrac{(N-1)N}{2} \right) = 2\left(1 - \chi\left(1 - \frac{1}{N}\right)\right) = 2\left(1-\frac1N\right)(\chi_2(N) - \chi).
\ee
Hence the evolution of the second moment is linear with a negative slope, since by assumption $\chi > \chi_2(N)$, which clearly contradicts the fact that the maximal time of existence $T^*=\infty$, and therefore the solution $\bx$ exists only up to a finite time: $T^*< \infty$. At exactly that time, only two things may happen: either the norm of the solution equals $+\infty$, i.e., $|\bx|_w = +\infty$, or two or more particles collide. The first option is not plausible since trivially the second moment of an empirical measure is finite at all times. We are thus only left with the collision of particles, that is $\bx$ has to blow up in finite time.
\end{proof}

\subsubsection{Simulations}
\label{subsubsec:simulations}

We give here a few simulations for the modified Keller-Segel equation showing various blow-up characteristics when $V=0$. As we want to capture the blow-up we used an adaptive time-step size as follows. For every time step $n\in \{0,\dots,M-1\}$, suppose we have a time-step size $\Delta_n t$, and compute the velocity $v_i^n$ of each particle $x_i^n$, $1\leq i \leq N$. Then fix a tolerance $\delta = 0.25$, and define
\bes
	 \delta_i = \begin{cases} \Delta_n t & \mbox{if $\Delta_n t \leq \frac{\delta \min_p(\Delta x_i^n, \Delta x_{i+1}^n)}{|v_i^n|}$},\\ \frac{\delta \min_p(\Delta x_i^n, \Delta x_{i+1}^n)}{|v_i^n|} & \mbox{otherwise}. \end{cases}
\ees
Finally, renew $\Delta_n t := \min_{i\in\{1,\dots,N\}} \delta_i$, compute the positions $x_i^{n+1}$ with the new $\Delta_n t$, and start over until $\Delta_n t > 10^{-7}$; when $\Delta_n t \leq 10^{-7}$ stop the simulation. We took $\Delta_0 t = 10^{-5}$ as the very initial time-step size. In Figures \ref{fig:ks}, \ref{fig:blowup} and \ref{fig:ks-twobumps} the simulations shown stopped due to this adaptive time-step procedure.

Figures \ref{fig:ks} and \ref{fig:blowup} show the results of simulations with initial continuum profile $\rho_0^\mt{heat}$, with $I_\mt{init} = [-2.5,2.5]$. From Figure \ref{fig:ks} one can see that the scheme we used captures nicely the formation of the blow-up for a supercritical parameter $\chi$.

\bfig[!ht]
	\subcaptionbox{Early evolution.\label{fig:ks-evo}}{\includegraphics[scale=0.41]{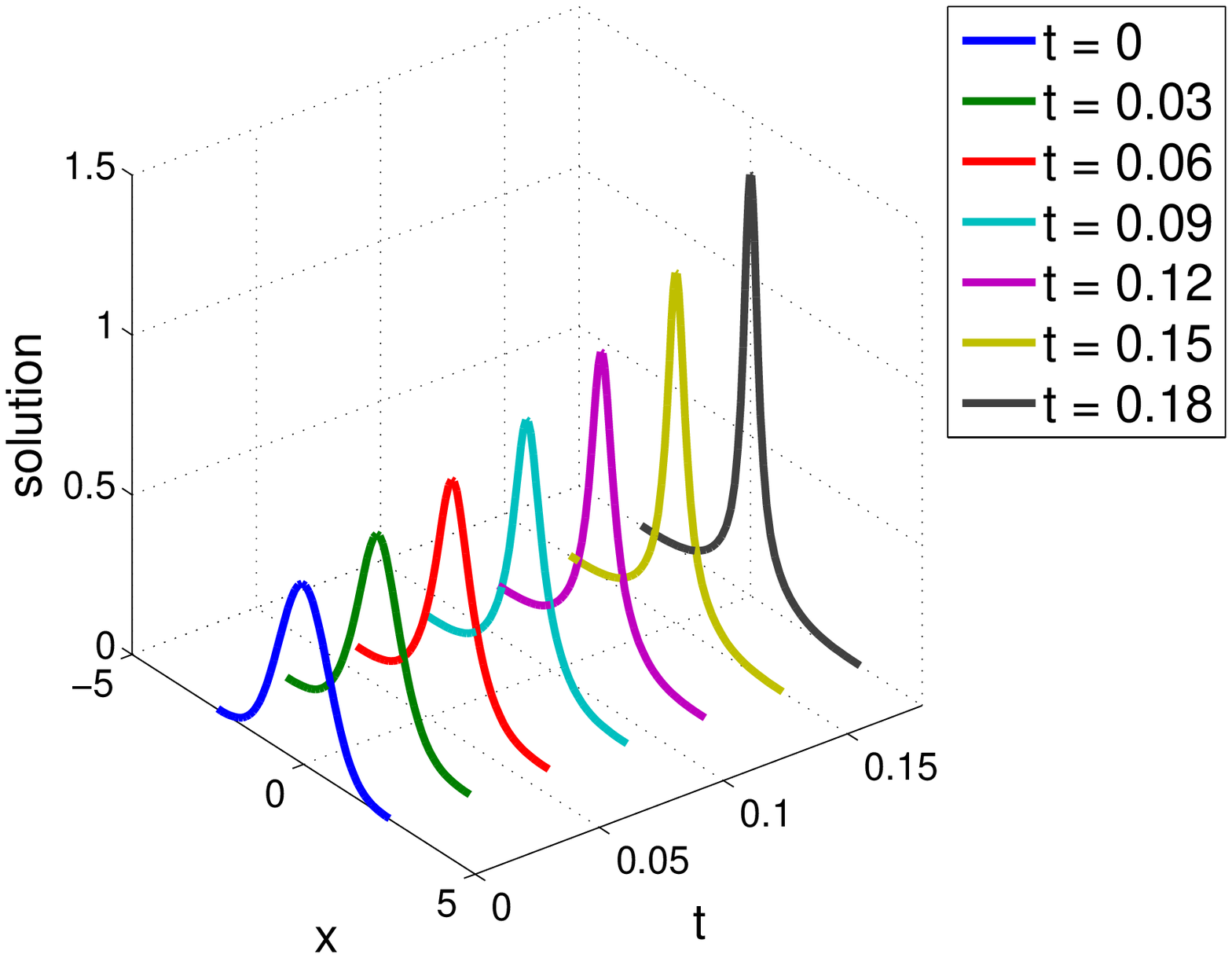}}
	\subcaptionbox{Blow-up formation.\label{fig:ks-blowup}}{\includegraphics[scale=0.41]{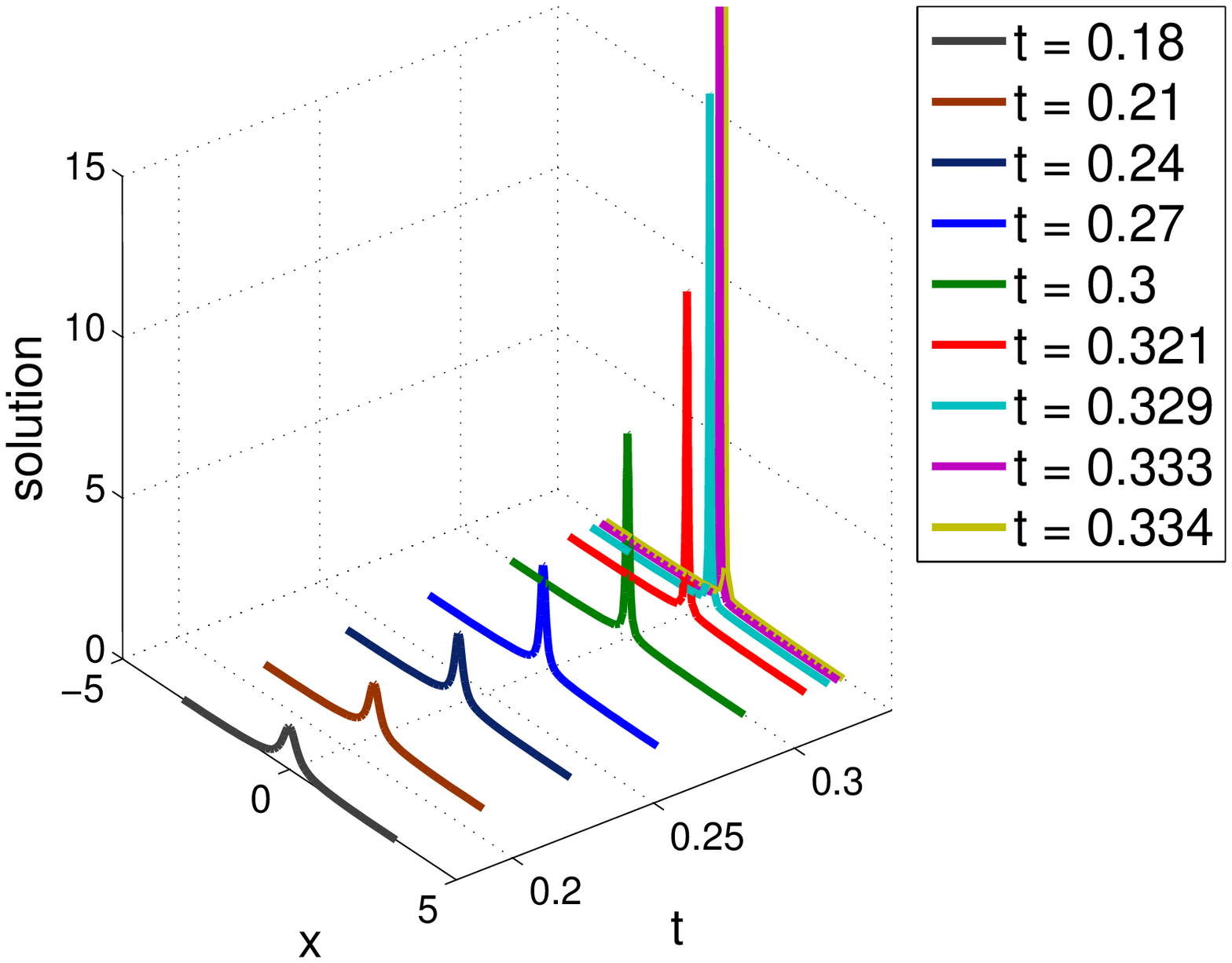}}
	\caption{The modified Keller-Segel equation with $\chi = 1.5$ for $N=50$.\label{fig:ks}}
\efig

\bfig[!ht]
\centering
	\subcaptionbox{Evolution of the second moment for $N = 100$.\label{fig:ks-moment}}{\includegraphics[scale=0.41]{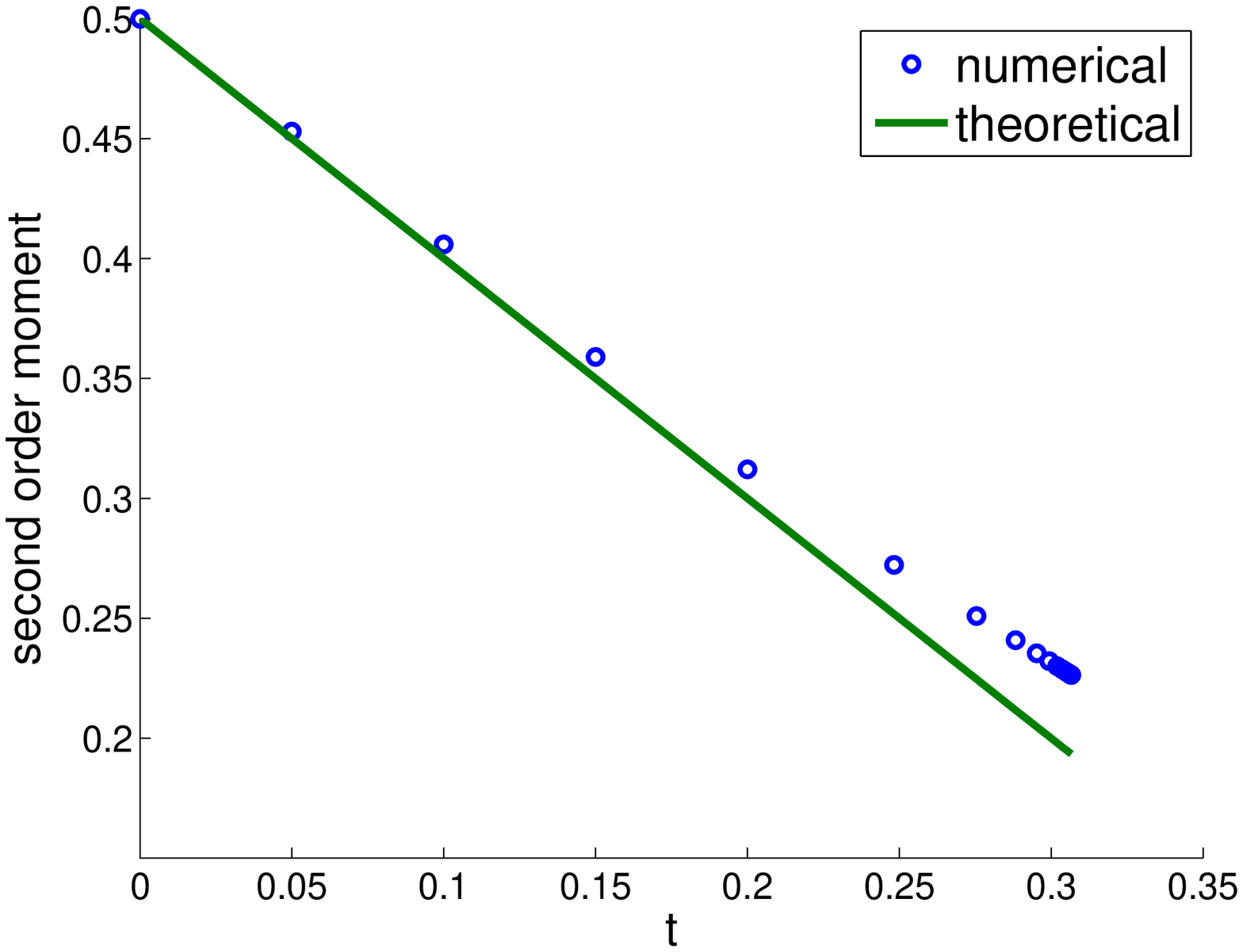}}
	\subcaptionbox{Particle trajectories up to first numerical blow-up for $N = 50$ (not all particles are represented).\label{fig:ks-positions}}{\includegraphics[scale=0.42]{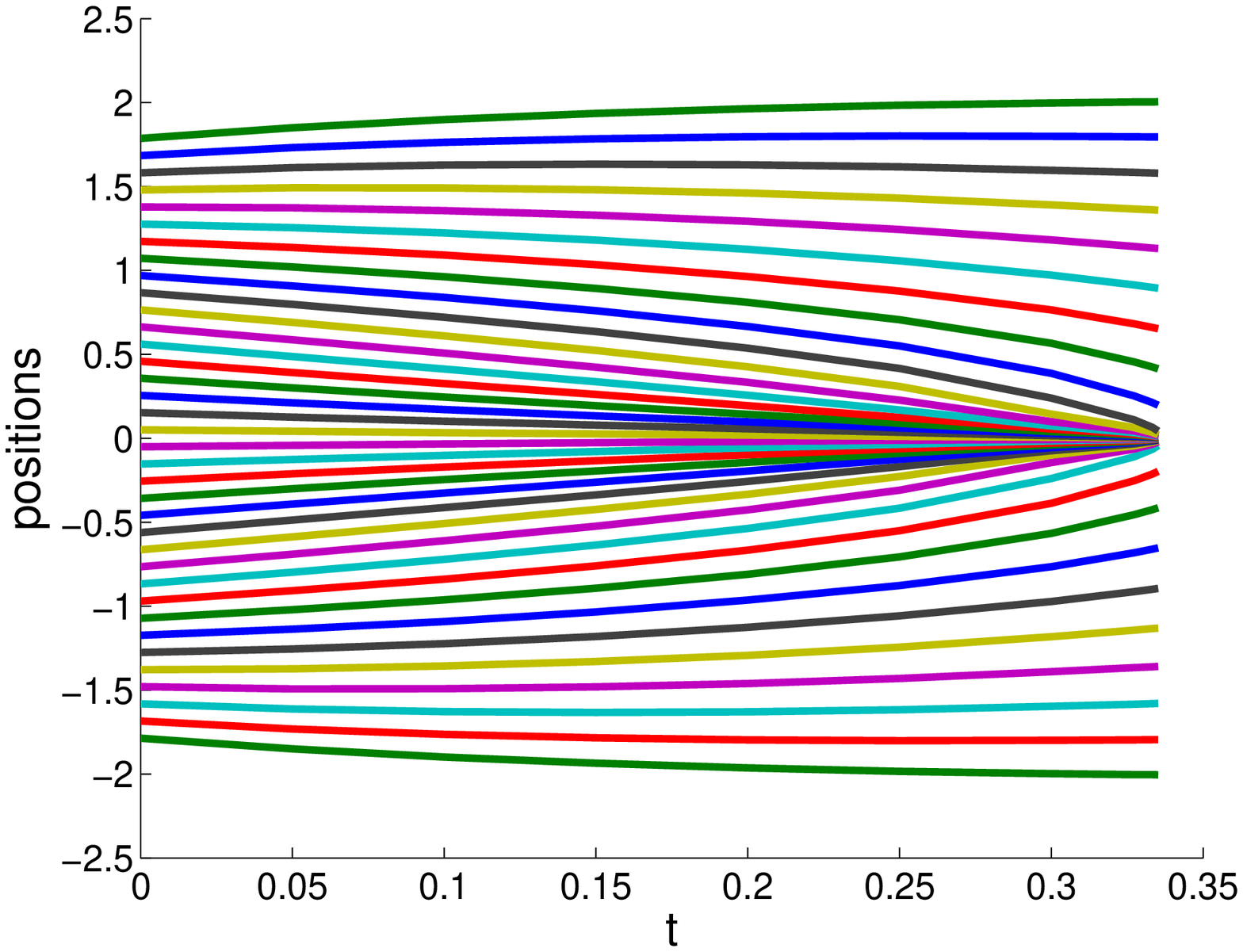}}
	\caption{The blow-up of the modified Keller-Segel equation with $\chi = 1.5$.\label{fig:blowup}}
\efig

Figure \ref{fig:ks-moment} shows that the evolution of the second moment is linear for some range of time with a slope that deviates slightly from the theoretical one, as expected from Proposition \ref{prop:keller-segel-blow-up} and its proof. Actually, by comparing \eqref{eq:cont-moment} and \eqref{eq:discrete-moment}, it can be seen that this slope deviation decreases as $N$ increases. For larger times, as the blow-up is approached and the time step refined, the numerical slope deviates even more from the theoretical one.

Figure \ref{fig:ks-positions} shows the trajectories of the particles up to the first blow-up time, defined numerically as the first time when the distance between two particles gets smaller than some chosen number $d_\mt{min}$, or equivalently when the adaptive time-step size $\Delta_nt$ gets larger than $10^{-7}$, see above. A possible procedure to continue the evolution after the first blow-up time is merging particles into a new heavier one whenever the distance between these particles gets less than a certain threshold, say proportional to $d_\mt{min}$; the weight of the new particle is then chosen to be the sum of the merged particles, and the position the barycentre of the merged particles. This procedure might give an idea of how the particles behave after the first blow-up, but we found that it is not very accurate since the post-collision trajectories strongly depended on the choice of the threshold, which is arbitrary. We found that the analysis of the post-collision behaviour is very delicate without having clear criteria for deciding when to merge particles and how many simultaneously. We thus leave this issue for further analysis and future work. 

Figure \ref{fig:ks-twobumps} shows the result with a continuum two-bump initial profile, with $I_\mt{init} = [-4.5,4.5]$:
\bes
	\rho_0(x) = \frac{1}{2\sqrt{4\pi t_0}}e^{-\frac{(x+2)^2}{4t_0}} + \frac{1}{2\sqrt{4\pi t_0}}e^{-\frac{(x-2)^2}{4t_0}} \quad \mbox{with $t_0=0.25$}.
\ees

\bfig[!ht]
\centering
	\subcaptionbox{Evolution with $\chi = 1.8$ for $N=50$.}{\includegraphics[scale=0.41]{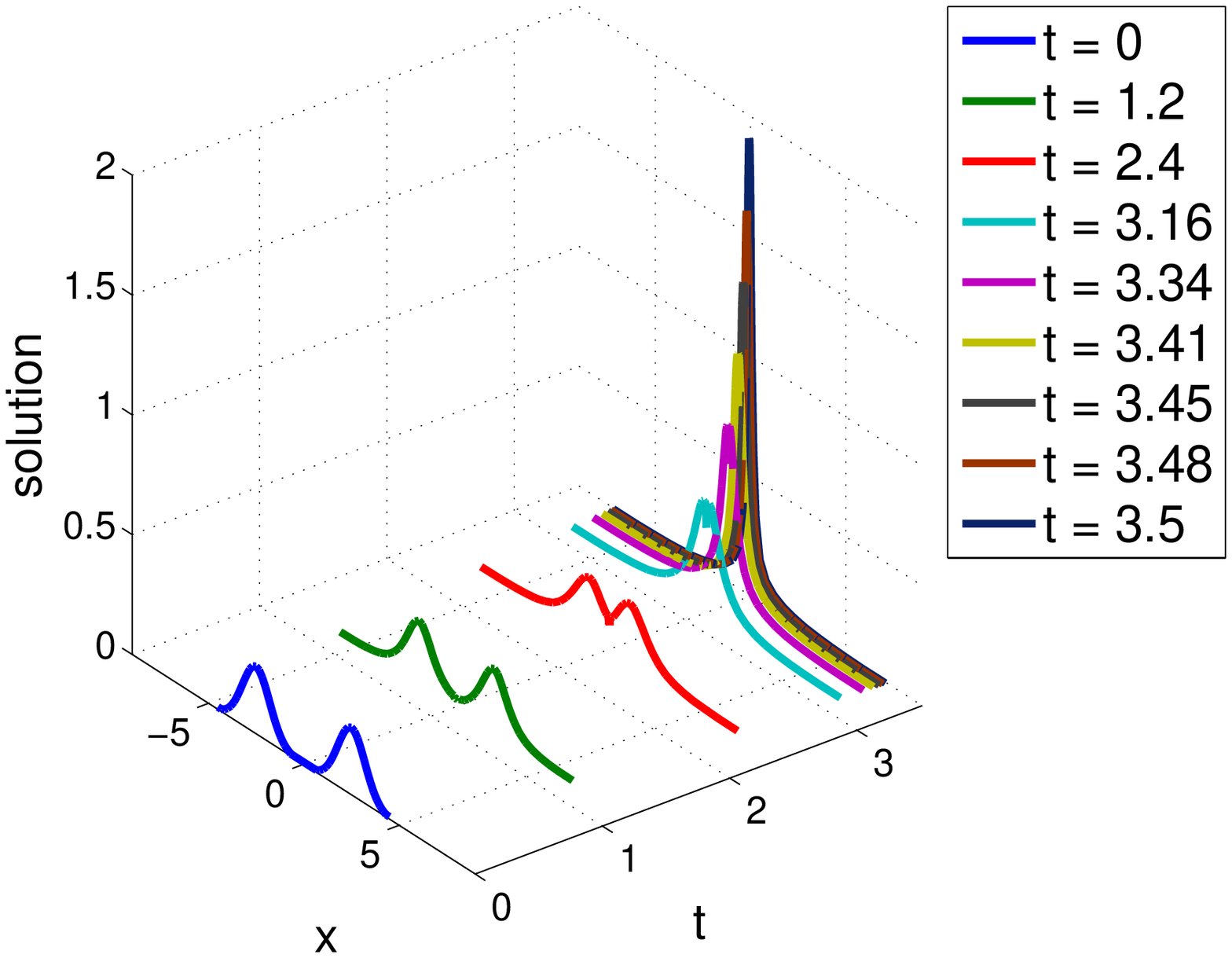}}
	\subcaptionbox{Evolution with $\chi = 3$ for $N=100$.\label{fig:ks-twobumps-distinct}}{\includegraphics[scale=0.41]{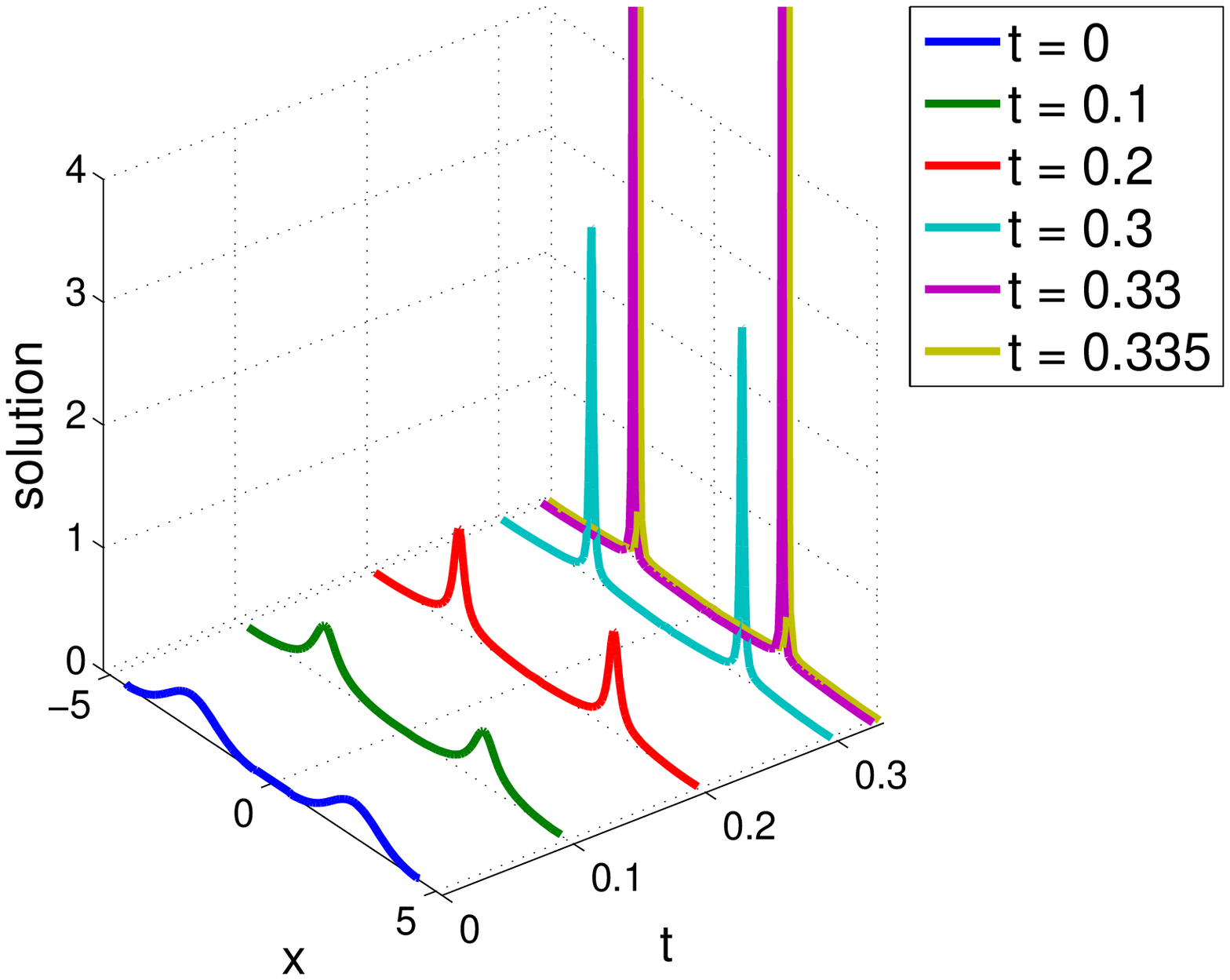}}
	\caption{Blow-up formation for the modified Keller-Segel equation with two initial Gaussian bumps.\label{fig:ks-twobumps}}
\efig

Figure \ref{fig:ks-twobumps} shows the possible formation of several Dirac masses, according to how much attraction is involved in the system and to how many bumps are present at the beginning. It seems like the more attraction, the more Dirac masses can form. Note that the two peaks in Figure \ref{fig:ks-twobumps-distinct} are actually of same height despite a displaying artifact; also, even if not clear from Figure \ref{fig:ks-twobumps-distinct}, the two peaks get slightly closer to each other with time, but then blow up before merging.

\subsection{The modified Keller-Segel equation with nonlinear diffusion}
\label{subsec:keller-segel-nl}

Let us now consider the modified one-dimensional Keller-Segel equation with nonlinear diffusion, i.e., the continuum gradient flow \eqref{eq:gradient-flow} with $H(\rho) = \frac{\rho^m}{m-1}$, $m>1$, $V = 0$ and $W(x) = 2\chi\log|x|$ (and $W(0):=0$). The initial continuum profile we used here is $\rho_0^\mt{heat}$, with $I_\mt{init} = [-2.5,2.5]$.

\bfig[!ht]
\includegraphics[scale=0.41]{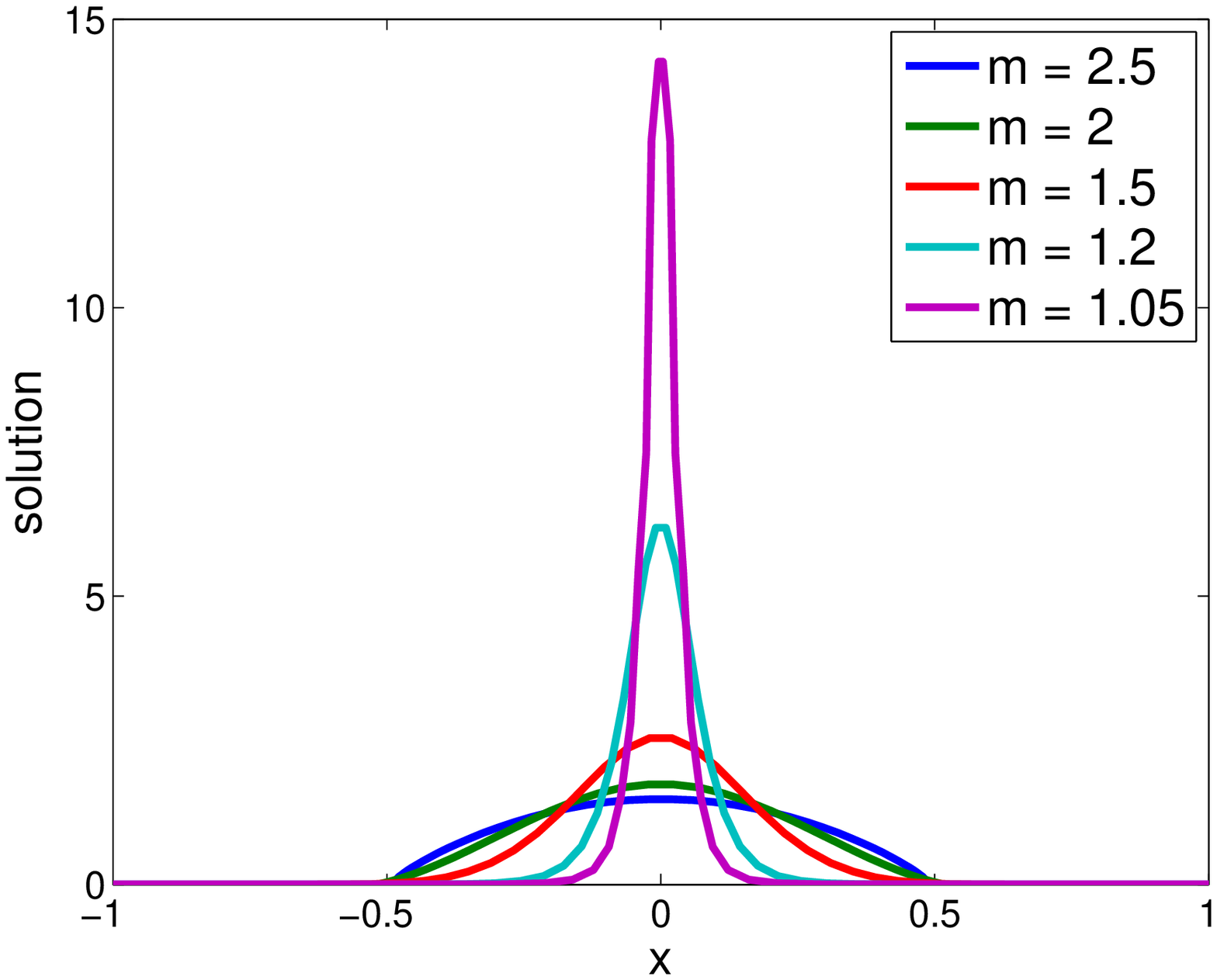}
\caption{The modified nonlinear Keller-Segel equation with $\chi = 1.4$ for different choices of $m$, for $N=50$ at $T=4$ with $\Delta t = 10^{-5}$.\label{fig:nlks}}
\efig

Each curve in Figure \ref{fig:nlks} is a good approximation of a steady state for the modified Keller-Segel equation with nonlinear diffusion. For each $m$, the steady state is different; as $m$ tends to $1$ the steady state ``squeezes'' and looks as if it is approaching a Dirac mass, which is the ``steady state'' of the modified Keller-Segel equation with linear diffusion studied in Section \ref{subsec:keller-segel}.

\subsection{A compactly supported potential with nonlinear diffusion}
\label{subsec:compact}

We consider here the continuum gradient flow when $H(\rho) = \frac{\rho^m}{m-1}$, $m>1$, $V=0$ and $W(x) = -c\max(1- |x|,0) + c$, $c > 0$, is a compactly supported interaction potential. In Figure \ref{fig:compact} the considered continuum initial profile is a uniform distribution on the interval $[-2,2]$. Here $I_\mt{init} = [-2,2]$ and the end particles were set to have weights equal to $0.001$.

\bfig[!ht]
\centering
	\subcaptionbox{Formation of metastable state.\label{fig:compact-evo1}}{\includegraphics[scale=0.41]{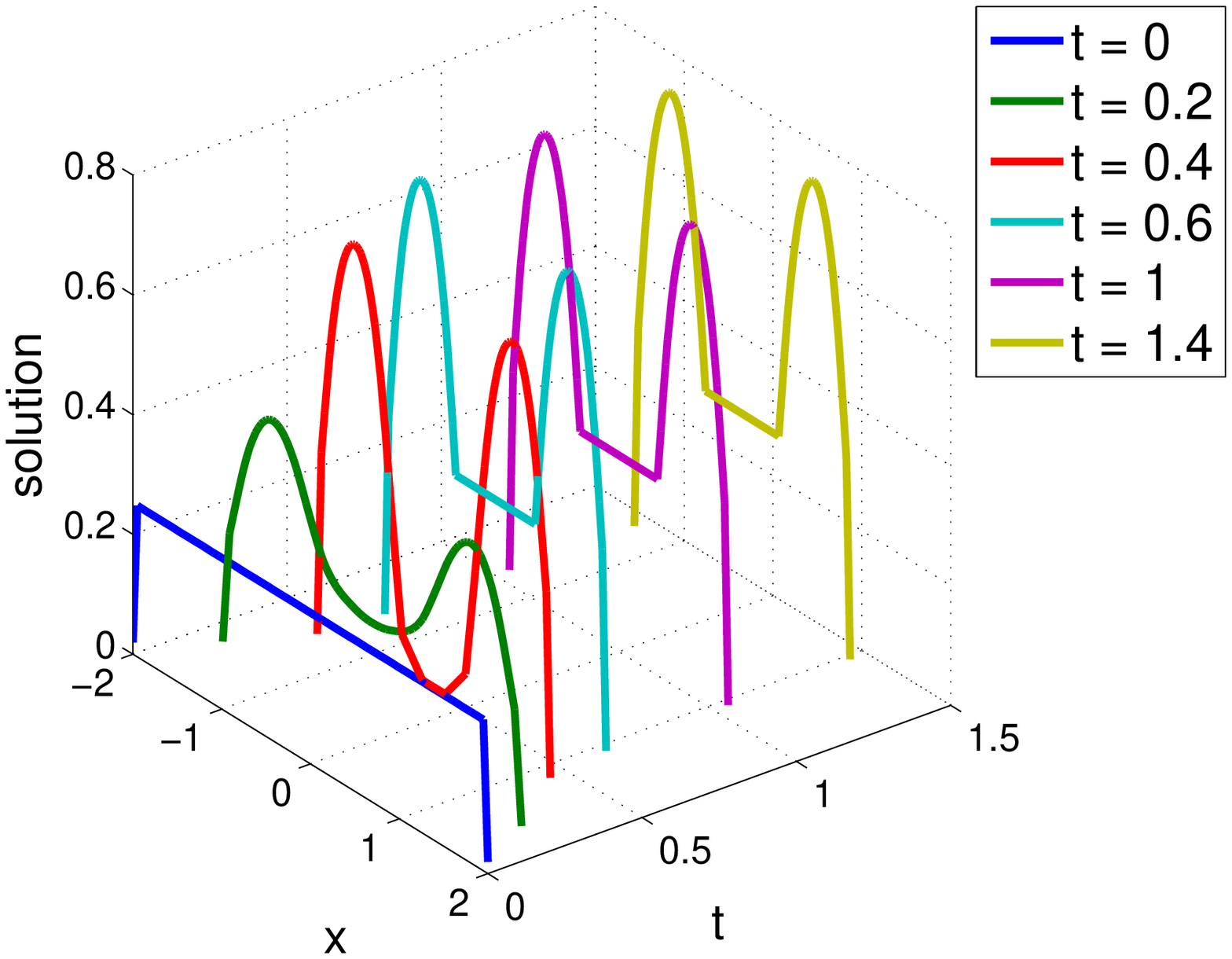}}
	\subcaptionbox{Formation of steady state.\label{fig:compact-evo2}}{\includegraphics[scale=0.41]{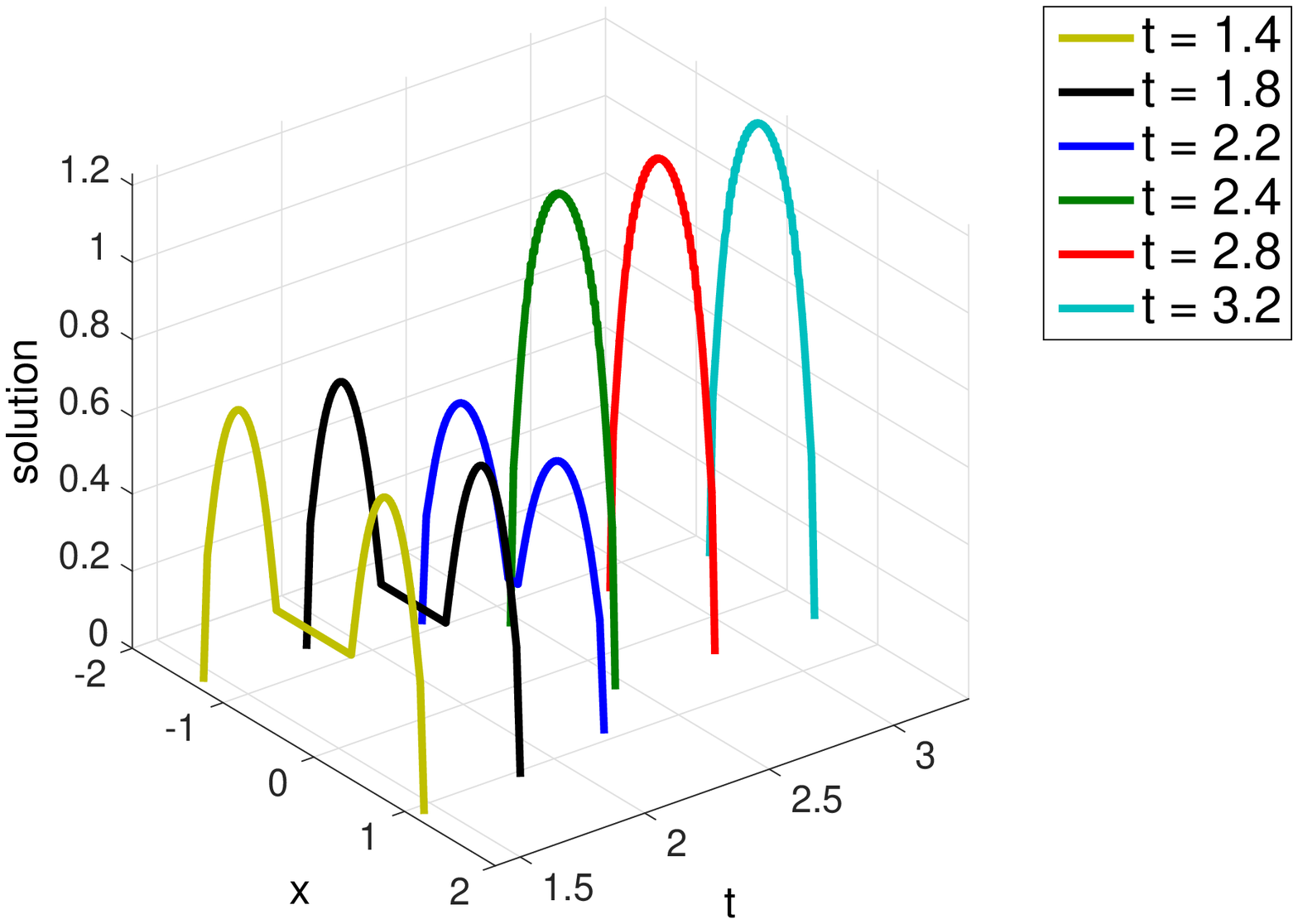}}
	\caption{Compactly supported potential $W(x) = -c\max(1-|x|,0) + c$ with nonlinear diffusion with $c = 8$ and $m = 3$, for $N=80$ with $\Delta t = 10^{-5}$.\label{fig:compact}}
\efig

Figure \ref{fig:compact-evo1} shows the formation of a metastable state made of two bumps, while Figure \ref{fig:compact-evo2} shows how this metastable state breaks into a single-bumped steady state. This behaviour, for exactly this interaction potential (up to a multiplicative constant, $c$), was already noted in \cite[Example 3]{CCH2} when using a finite-volume scheme on gradient flows. The bumps are actually supposed to be disconnected since no particles were found numerically in-between, although this does not seem to be the case in the plots. This is because the particle of each bump which is closest to the origin is not located at the actual boundary of the exact solution's corresponding bump, and so the line connecting the two bumps does not show $0$ density but more. Visually, a better separation of the bumps can be obtained by increasing the number of particles in the simulation.

\begin{appendix}

\section{Convergence of the $p$-approximated gradient flow in one dimension}
\label{app:convergence}

In this appendix we give the proof of the convergence of the $p$-approximated gradient flow \eqref{eq:discrete-gradient-flow-ode} to the discrete one \eqref{eq:discrete-gradient-flow-inclusion} in one dimension, see Section \ref{subsec:one-d}. Before doing so we need to recall some notions and results from monotone operator theory.

If $A\: \{y\in X \mid Ay \neq \emptyset\} \subset X\to X^*$, where $X$ is a Banach space and $X^*$ its dual, is a maximal monotone operator, then the \emph{graph} of $A$ is given by $\graph(A)=\{(x,Ax) \mid x \in \{y\in X \mid Ay \neq \emptyset\}\}$.
\begin{defn}[Graph convergence] \label{defn:graph}
	Let $X$ be a reflexive Banach space, $(A_k)_{k\in\N}$ a sequence of maximal monotone operators from $X$ to $X^*$, and $A$ a maximal monotone operator from $X$ to $X^*$. We say that $(A_k)_{k\in\N}$ \emph{graph-converges} to $A$ if, for every $(x,y) \in \graph(A)$, there exists a sequence $((x_k,y_k))_{k\in\N}$ with $(x_k,y_k) \in \graph(A_k)$ for all $k\in\N$ such that $x_k \to x$ strongly in $X$ and $y_k \to y$ strongly in $X^*$ as $k\to\infty$.
\end{defn}
\begin{rem}
	One can check that if $\phi$ is a lower semi-continuous, convex function from a reflexive Banach space to $\R$, then $\p \phi$ is a maximal monotone operator, see \cite[Section 3.7.1]{Attouch}.
\end{rem}

For the sake of completeness we give the defintion of $\Gamma$-convergence.
\begin{defn}[$\Gamma$-convergence] \label{defn:gamma-convergence}
	Let $(X_k)_{k\in\N}$ be a sequence of metric spaces endowed with a distance $d$ and $(\phi_k)_{k\in\N}$ be a sequence of functionals $\phi_k \: X_k \to \R$ for all $k\in\N$. We say that $(\phi_k)_{k\in\N}$ \emph{$\Gamma$-converges} to $\phi$ if the following two conditions are met for all $u \in X$:
\begin{enumerate}[label=(\roman*)]
	\item \label{cond:liminf} (``liminf" condition) $\phi(u) \leq \liminf_{k \to \infty} \phi_k(u_k)$ for all sequences $(u_k)_{k\in\N}$ with $u_k \in X_k$ for all $k\in\N$ and $d(u_k,u) \to 0$ as $k \to \infty$,
	\item \label{cond:limsup} (``limsup" condition) $\limsup_{k \to \infty} \phi_k(u_k) \leq  \phi(u)$ for some sequence $(u_k)_{k\in\N}$ with $u_k \in X_k$ for all $k\in\N$ and $d(u_k,u) \to 0$ as $k \to \infty$.
\end{enumerate}
\end{defn}

Theorem \ref{thm:gammatograph} connects the notion of $\Gamma$-convergence to that of graph-convergence in finite dimension. It is a consequence of \cite[Theorem 3.66]{Attouch} and the fact that $\Gamma$-convergence is equivalent to Mosco convergence in finite dimension. Indeed, in a general dimensional setting, Mosco convergence means that the ``liminf" condition of $\Gamma$-convergence holds for the weak topology and the ``limsup" condition holds for the strong topology, see \cite[Definition 2.2]{Mielke} and \cite[Definition 3.17]{Attouch}.
\begin{thm} \label{thm:gammatograph}
	Let $X$ be a finite-dimensional Banach space. Let $(\phi_k)_{k\in\N}$ be a sequence of lower semi-continuous, convex functions with $\phi_k \: X \to \R$ for all $k\in\N$, and $\phi\: X \to \R$ a lower semi-continuous, convex function. If $(\phi_k)_{k\in\N}$ $\Gamma$-converges to $\phi$, then $(\p \phi_k)_{k\in\N}$ graph-converges to $\p \phi$.
\end{thm}

We can now give the general convergence and regularity result, whose proof can be deduced by \cite[Theorem 3.74]{Attouch}, \cite[Theorem 3.1]{Brezis} and \cite[Theorem 1, Section 3.2]{Aubin}.
\begin{thm} \label{thm:graph}
	Let $X$ be a Hilbert space.  Let $(\phi_k)_{k\in\N}$ be a sequence of lower semi-continuous, convex functions with $\phi_k \: X \to \R$ for all $k\in\N$ and $\phi\: X \to \R$ be a lower semi-continuous, convex function. Suppose that $(\p \phi_k)_{k\in\N}$ graph-converges to $\p \phi$. Consider the following differential inclusions, for all $k\in\N$.
	\bes
		u_k'(t) \in - \p \phi_k(u_k(t)), \quad u_k(0) = u_k^0 \quad \mbox{for almost every $t \in (0,T]$},
	\ees
and
	\bes
		u'(t) \in - \p \phi(u(t)), \quad u(0) = u_0 \quad \mbox{for almost every $t \in (0,T]$},
	\ees
where $u,u_k\:[0,T] \to X$ are the unknown curves. Unique solutions $u$ and $u_k$ exist and
\begin{enumerate}[label=(\arabic*)]
\item \label{it:cont} $u_k$ and $u$ are continuous on $[0,T]$,\\[-0.3cm]
\item \label{it:rc} $u_k'$ and $u'$ are right-continuous on $[0,T]$.
\end{enumerate}
Moreover, assume that $u_k^0 \to u_0$ strongly and that in this case, $\phi_k(u_k^0) \to \phi(u_0)$ as $k\to\infty$. Then
\begin{enumerate}[label=(\arabic*),resume]
\item \label{it:un} $u_k \to u$ uniformly on $[0,T]$ as $k \to \infty$,\\[-0.3cm]
\item \label{it:integral} $\int_0^T t\left|u_k'(t) - u'(t)\right|^2 \d t \to 0$ as $k \to \infty$,\\[-0.3cm]
\item \label{it:dun} $u_k' \to u'$ strongly in $L^2([0,T],X)$ as $k\to\infty$ (so $u_k'(t) \to u'(t)$ for almost every $t \in [0,T]$),\\[-0.3cm]
\item \label{it:phin} $\phi_k(u_k) \to \phi(u)$ uniformly on $[0,T]$ as $k\to\infty$.
\end{enumerate}
\end{thm}

It is not hard to see that $\wtE^p$, see \eqref{eq:energy-p-one-d}, $\Gamma$-converges to $\wtE$, see \eqref{eq:energy-discrete-particles}, as $p\to\infty$. Furthermore, it is easily verified that $\wtE^p(\boldsymbol{x^0}) \to \wtE(\boldsymbol{x^0})$ as $p\to\infty$, where $\boldsymbol{x^0} \in \R_w^N$ is taken here to be the initial condition for both the discrete and $p$-approximated discrete gradient flows \eqref{eq:discrete-gradient-flow-inclusion} and \eqref{eq:discrete-gradient-flow-ode}. Therefore, in order to use Theorems \ref{thm:gammatograph} and \ref{thm:graph} combined, we are only left with checking that $\wtE^p$ and $\wtE$ are lower semi-continuous and convex. The first condition is trivial to verify based on the assumptions on $H,V$ and $W$, whereas the convexity condition is shown below whenever $V$ and $W$ are assumed to be convex. Actually the following proposition shows the convexity of $\wtE$ only, but also holds for $\wtE^p$ since the proof is easily adapted from the classical minimum function to the $p$-approximated one.
\begin{prop} \label{prop:convexity-EN}
	Let $d=1$ and the confinement and interaction potentials $V$ and $W$ be convex. Then the discrete energy $\wtE$ defined in \eqref{eq:energy-discrete-particles} is convex.
\end{prop}
\begin{proof}
Let $\lambda \in [0,1]$ and $\bx,\by \in \R_w^N$. Then, by the facts that $\min$ is concave on $\R^2$ and $h$ is non-increasing and convex on $(0,\infty)$, we know that $h\circ\min$ is convex on $(0,\infty)^2$, where $\circ$ is the composition operator. Define, for all $a,b\in\R$, $[a,b]_\lambda = \lambda a + (1-\lambda)b$, and $r_i(\bx) = \min(\Delta x_i,\Delta x_{i+1})$ and $r_i(\by) = \min(\Delta y_i,\Delta y_{i+1})$. Therefore, since $V$ and $W$ are convex,
\begin{align*}
	\wt{E}_N([\bx,\by]_\lambda) &= \sum_{i = 1}^N w_i h\left(\frac{1}{w_i}\min( [\Delta x_i,\Delta y_i]_\lambda,
[\Delta x_{i+1},\Delta y_{i+1}]_\lambda)
\right)\\
	&\phantom{={}} + \sum_{i=1}^N w_i V([x_i,y_i]_\lambda) + \dfrac{1}{2} \sum_{i=1}^N \sum_{\substack{j=1\\j\neq i}}^N w_i w_j W([x_i-x_j,y_i-y_j]_\lambda)\\
	 &\leq \lambda \sum_{i = 1}^N w_i h\left(\frac{r_i(\bx)}{w_i}\right) + (1 - \lambda) \sum_{i = 1}^N w_i h\left(\frac{r_i(\by)}{w_i}\right) + \lambda \sum_{i=1}^N w_i V(x_i) + (1-\lambda)\sum_{i=1}^N w_i V(y_i)\\
	&\phantom{={}}+ \dfrac{\lambda}{2} \sum_{i=1}^N \sum_{\substack{j=1\\j\neq i}}^N w_i w_j W(x_i - x_j) + \frac{1-\lambda}{2} \sum_{i=1}^N \sum_{\substack{j=1\\j\neq i}}^N w_i w_j W(y_i - y_j)\\
	&= \lambda\wt{E}_N(\bx) + (1-\lambda)\wt{E}_N(\by).
\end{align*}
Hence convexity of $\wt{E}_N$.
\end{proof}
Theorem \ref{thm:gammatograph} now tells us that $\p_w\wtE^p$ graph-converges to $\p_w\wtE$ as $p\to\infty$, and Theorem \ref{thm:graph} tells us in which sense the $p$-approximated discrete gradient flow converges to the discrete one and also gives us some regularity on the discrete and $p$-approximated discrete gradient flow solutions. The use of the $p$-approximated gradient flow \eqref{eq:discrete-gradient-flow-ode} is therefore justified to approximate \eqref{eq:discrete-gradient-flow-inclusion} (at least in the case when $V$ and $W$ are convex and $d=1$).

\end{appendix}

\subsection*{Acknowledgements}

JAC, YH and FSP are supported by Engineering and Physical Sciences Research Council grant EP/K008404/1. JAC is also supported by the Royal Society through a Wolfson Research Merit Award. GW is supported by ISF grant 998/5.

\bibliography{jko_scheme}
\bibliographystyle{abbrv}

\end{document}